\documentclass[a4paper,12pt]{amsart}
\makeatletter
\def\@tocline#1#2#3#4#5#6#7{\relax
	\ifnum #1>\c@tocdepth 
	\else
	\par \addpenalty\@secpenalty\addvspace{#2}%
	\begingroup \hyphenpenalty\@M
	\@ifempty{#4}{%
		\@tempdima\csname r@tocindent\number#1\endcsname\relax
	}{%
		\@tempdima#4\relax
	}%
	\parindent\z@ \leftskip#3\relax \advance\leftskip\@tempdima\relax
	\rightskip\@pnumwidth plus4em \parfillskip-\@pnumwidth
	#5\leavevmode\hskip-\@tempdima
	\ifcase #1
	\or\or \hskip 1em \or \hskip 2em \else \hskip 3em \fi%
	#6\nobreak\relax
	\dotfill\hbox to\@pnumwidth{\@tocpagenum{#7}}\par
	\nobreak
	\endgroup
	\fi}
\makeatother
\usepackage[margin=2cm]{geometry}
\usepackage{amssymb,amsmath,latexsym,amsthm,mathrsfs} 
\usepackage{aliascnt}
\usepackage{resmes}
\usepackage{graphicx}
\usepackage{enumerate}
\usepackage{comment}
\usepackage{aurical}
\usepackage[hidelinks]{hyperref} %

\usepackage{xcolor}
\definecolor{dark_orange}{HTML}{CC6633}

\usepackage[autostyle,italian=guillemets]{csquotes}

\numberwithin{equation}{section}

\newtheorem{theorem}{Theorem}[section]
\newtheorem{lemma}[theorem]{Lemma}
\newtheorem{cor}[theorem]{Corollary}
\newtheorem{prop}[theorem]{Proposition}
\theoremstyle{definition}
\newtheorem{definition}[theorem]{Definition}
\newtheorem{rem}[theorem]{Remark}
\newtheorem{example}[theorem]{Example}

\newcommand{\R}{\mathbb{R}}
\newcommand{\N}{\mathbb{N}}
  
\renewcommand{\S}{\mathbb{S}}  

\newcommand{\norm}[1]{\|#1\|}
\newcommand{\haus}[1]{\mathcal{H}^{#1}}
\newcommand{\cont}{\mathcal{C}}
\renewcommand{\phi}{\varphi}

\newcommand{\subgr}{\mbox{subgr}}

\newcommand{\Per}{\operatorname{Per}} 

\newcommand{\loc}{\rm{loc}}
\renewcommand{\leq}{\leqslant}

\renewcommand{\geq}{\geqslant}

\renewcommand{\epsilon}{\varepsilon}
\allowdisplaybreaks

\title[Regularity theory for non-local almost minimal surfaces]{Regularity theory for\\ non-local almost minimal surfaces\\
via Euler-Lagrange inequalities at small scales}
\date{}

\author{\href{https://research-repository.uwa.edu.au/en/persons/serena-dipierro}{Serena Dipierro}}
\address{
	Serena Dipierro,
	Department of Mathematics and Statistics, 
	The University of Western Australia, 
	35 Stirling Highway, Perth, WA 6009, Australia}
\email{serena.dipierro@uwa.edu.au}

\author{Ovidiu Savin}
\address{
	Ovidiu Savin,
	Department of Mathematics, 
	Columbia University,
	2990 Broadway,	New York, NY 10027, United States of America}
\email{savin@math.columbia.edu}

\author{\href{https://research-repository.uwa.edu.au/en/persons/enrico-valdinoci}{Enrico Valdinoci}}
\address{
	Enrico Valdinoci,
	Department of Mathematics and Statistics, 
	The University of Western Australia, 
	35 Stirling Highway, Perth, WA 6009, Australia}
\email{enrico.valdinoci@uwa.edu.au}

\author{\href{https://research-repository.uwa.edu.au/en/persons/riccardo-villa-2/}{Riccardo Villa}}
\address{
	Riccardo Villa,
	Department of Mathematics and Statistics, 
	The University of Western Australia, 
	35 Stirling Highway, Perth, WA 6009, Australia}
\email{riccardo.villa@research.uwa.edu.au}

\begin{document}
	
	\maketitle
	\begin{abstract}		
		We extend the notion of non-local almost minimal sets introduced in [Dipierro, Valdinoci, and Villa, \emph{J. Math. Pures Appl.}~(2026)] by admitting lower-order powers for the remainder term.
		
		We combine classical techniques and a new, weaker version of the Euler-Lagrange inequalities (at a microscopic scale) to obtain $\cont^{1,\alpha}$-regularity for non-local almost minimal sets, up to Hausdorff dimension $n-3$.
		
		In turn, this formulation is well-suited for the non-local counterpart of the so-called Massari's problem for sets with prescribed mean curvature. 
		In particular, such a generalization allows us to consider sets with $L^p$ non-local mean curvature (with $p\geq n/s$), which do not satisfy any curvature equation (neither in the classical nor in the viscosity sense).
		
The results presented here lie beyond the reach of previously established methodologies. Indeed, the approach developed herein is based on a conceptual shift relative to the existing literature, which relied on a strong form of the Euler-Lagrange equation that is not available for general non-local almost minimal sets.
	\end{abstract}
	
	\setcounter{tocdepth}{2}
	\tableofcontents

	\section{Introduction and main results}\label{sec::intro}
	
	In a recent work, the first and the last two authors have addressed a non-local version of the so-called Massari's problem about sets with prescribed mean curvature. In turn, they have shown regularity of (a tailored version of) almost minimizers for the fractional perimeter (see \cite{MR5015001}).
	
	In the present paper, we develop a more comprehensive regularity theory for almost minimal sets to include a wider class of almost minimizers for the $s$-perimeter, i.e. minimizers for the fractional perimeter up to a remainder of order less than $n-s$.
	We investigate non-local almost minimal sets as mathematical objects that are interesting in their own right, developing a framework that covers interesting applications.
	
	Indeed, non-local almost minimizers theory finds broad utility in a variety of physical and geometric problems such as obstacle and free boundary problems, capillarity problems, and regularity of sets with prescribed non-local mean curvature (see for instance~\cite{MR3532394, MR3707346, MR3717439}). 
	
	Another interesting setting in which almost minimizers naturally appear is given by metastable configurations. When a phase-coexistence model or an interface dynamics system is thrown out of its ground state (e.g., by an oscillatory external force or a mass constraint), the resulting configurations do not satisfy a classical stationary equation. Instead, they may become stuck in a near-stationary regime for a long time. Therefore, it is natural to interpret these intermediate states as almost minimal for the functional under consideration, and investigate their stability and regularity.
	\medskip
	
	We stress that both the primary novelty and the main mathematical challenge of the present work lie in allowing for a parameter $\gamma\in(0,s)$. Such a relaxation requires a conceptual shift from the existing literature and has significant implications for the techniques employed to establish our regularity result. 
	Indeed, almost minimizers are often adopted to study proper minimizers, allowing one to interpret the minimizers of a ``complicated''  energy functional as almost minimizers of a ``simpler'' one, for which a well-developed regularity theory exists. On the other hand, almost minimizers do not generally satisfy an Euler-Lagrange equation, thereby losing global control on the fractional mean curvature $H_s$. 
	
	However, through a careful analysis and leveraging the lower-order scaling of the error term, we deduce here a bound on $H_s$ at a microscopic scale and we establish a weaker version of the Euler-Lagrange inequalities (compare Theorem~\ref{th::EL} here with \cite[Theorem 2.6]{MR5015001}).	
	\bigskip
	
	We start our discussion by recalling the definition of non-local interaction between disjoint sets introduced in~\cite{MR2675483}. Whenever needed, all subsets of~$\R^n$ will be implicitly assumed to be Lebesgue-measurable.
	
	For every couple of disjoint sets $A, B\subseteq\R^n$, let the non-local interaction between $A$ and $B$ be
	\begin{equation*}
		\mathcal{L}_s(A,B) := \iint_{A\times B} \frac{dx\,dy}{|x-y|^{n+s}}.
	\end{equation*}
	
	Given an open bounded Lipschitz domain $\Omega\subseteq\R^n$ and a set $E\subseteq\R^n$, we recall that
	the (localized) $s$-perimeter of $E$ in $\Omega$ is defined in \cite{MR2675483} as
	\begin{equation*}
		\begin{split}
			\Per_s(E,\Omega) &= \mathcal{L}_s(E\cap\Omega,E^c\cap\Omega) + \mathcal{L}_s(E\cap\Omega,E^c\cap\Omega^c) + \mathcal{L}_s(E\cap\Omega^c,E^c\cap\Omega) \\
			& = \iint_{\Omega_\sharp} \frac{\chi_E(x)\chi_{E^c}(y)}{|x-y|^{n+s}}dx\,dy .
		\end{split}
	\end{equation*}
	where $\Omega_\sharp:=\R^n\times\R^n\setminus(\Omega^c\times\Omega^c)$.
	
	\begin{definition} \label{def::almost_min}
		Let $\Omega\subseteq\R^n$ be an open bounded Lipschitz domain. Let $\Lambda\geq 0$ and $\gamma\in(0,s]$.
		
		We say that a set $E\subseteq\R^n$ is $(\Lambda,\gamma)$-almost minimal in $\Omega$ if 
		\begin{equation}\label{eq::almost_min}
			\Per_s(E,\Omega) \leq \Per_s(F,\Omega)+\Lambda r^{n-s+\gamma},
		\end{equation}
		for any set $F\subseteq\R^n$ such that $F\Delta E\subseteq \widetilde{B}_r$, for some ball $\widetilde{B}_r \Subset \Omega$.
		%
		When $\Omega$ is unbounded, we say that $E\subseteq\R^n$ is $(\Lambda,\gamma)$-almost minimal in $\Omega$ if it is $(\Lambda,\gamma)$-almost minimal in $\Omega'$, for any $\Omega'\Subset\Omega$.
	\end{definition}
	
	\begin{rem} \label{rem::massari}
		Let $s\in(0,1)$, $p>\frac{n}{s}$, and $H\in L^p(B_1)$. We recall that, for any set $E\subseteq\R^n$, the the non-local Massari's functional is defined in \cite{MR5015001} as
		\begin{equation*}
			\mathcal{J}_{s,H}(E,B_1) := \Per_s(E,B_1) + \int_{E\cap B_1} H(x)dx.
		\end{equation*}
		
		By a direct computation, we have that a minimizer of $\mathcal{J}_{s,H}$ is a $(\Lambda,\gamma)$-almost minimal in the sense of Definition~\ref{def::almost_min} with $\Lambda := \norm{H}_{L^p(B_1)}$ and $\gamma :=s-\frac{n}{p}$. 
	\end{rem}
	
	In the wake of \cite{MR5015001}, we also give the definitions of almost sub- and almost supersolution property.
	\begin{definition} \label{def::sub_super_sol}
		Let $\Omega\subseteq\R^n$ be an open bounded Lipschitz domain. Consider $\Lambda>0$ and $\gamma\in(0,s]$. 
		
		We say that~$E$ satisfies the~$(\Lambda,\gamma)$-almost subsolution property if
		\begin{equation} \label{eq::sub_sol_prop}
			\Per_s(E\cup A,\Omega)-\Per_s(E,\Omega)\mathcal={L}_s(A,E\setminus A)-\mathcal{L}_s(A,E^c)\geq -\Lambda r^{n-s+\gamma},
		\end{equation}
		for every~$A\subseteq E\cap \widetilde{B}_r$, for some ball $\widetilde{B}_r \Subset \Omega$.
		
		Similarly, we say that a set $E\subseteq\R^n$ satisfies the $(\Lambda,\gamma)$-almost supersolution property in $\Omega$ if
		\begin{equation}  \label{eq::super_sol_prop}
			\Per_s(E,\Omega)-\Per_s(E\cup A,\Omega)=\mathcal{L}_s(A,E)-\mathcal{L}_s(A,E^c\setminus A)\leq \Lambda r^{n-s+\gamma},
		\end{equation}
		for every~$A\subseteq E^c\cap \widetilde{B}_r$, for some ball $\widetilde{B}_r \Subset \Omega$.
	\end{definition}
	
	\begin{rem}
		Recalling \cite[Lemma 1.3]{MR5015001}, we have that an almost minimal set satisfies both the almost sub- and the almost supersolution property.
		
		Also, if a set satisfies both the $(\Lambda,\gamma)$-almost subsolution property and the $(\Lambda,\gamma)$-almost supersolution property, then it is $(2\Lambda,\gamma)$-almost minimal. 
	\end{rem}
	\medskip
	
	We point out that, in the classical case, other definitions of almost minimality are available.
	Compare for instance, Definition~\ref{def::almost_min} with the definition of almost minimality coming from capillary problems \cite[Definition 1.5]{MR3717439} (see also \cite[Definition 1.1]{MR5015001}) or Almgren's definition in~\cite{MR420406} (see also~\cite[Definition 2.10]{MR2518104}), in which the almost minimality is encode in a multiplicative term of lower order. This notion is also known as \emph{quasi minimality}. 
	A key advantage of our approach is that we consider a weaker definition that encompasses the others as special cases. This provides a unified framework for the topic and isolates its essential features.
	
	\begin{definition}[Quasi-minimality] \label{def::almgren_def}
		Let $K\geq0$ and some $\beta\in(0,s]$. A set $E\subset\R^n$ is said to be $(K, \beta)$-quasi minimal in $\Omega$ if, for every ball $\widetilde{B}_r\Subset\Omega$ of radius $r$ and for every set $F$ such that $F\setminus \widetilde{B}_r= E\setminus \widetilde{B}_r $, we have
		\begin{equation} \label{eq::almgren_def}
			\Per_s(E,\widetilde{B}_r) \leq (1+Kr^\beta)\Per_s(F,\widetilde{B}_r).
		\end{equation}
		When a set $E$ is $(K,0)$-quasi minimal, we say that $E$ is $Q$-quasi minimal, for some $Q>1$.
	\end{definition}
	\smallskip
	
	We stress every quasi minimal set in the sense of Definition~\ref{def::almgren_def} is also almost minimal in the sense of Definition~\ref{def::almost_min}, according to the following:
	\begin{lemma} \label{lemma::quasi_are_almost}
		Let $E$ be a $(K,\beta)$-quasi minimal set in $B_1$. Then, $E$ is $(\Lambda,\gamma)$-almost minimal in $B_1$, with $\gamma=\beta$ and for some positive constant $\Lambda=\Lambda(n,s,K)$.
	\end{lemma}
	We present a proof of Lemma~\ref{lemma::quasi_are_almost} in Section~\ref{sec::quasi_are_almost}.

	\subsection{Regularity for almost minimal surfaces}
	We show that almost minimal sets have $\cont^{1,\alpha}$-regular boundary, for some small $\alpha\in(0,1)$.
	The precise statement goes as follows:
	
	\begin{theorem} \label{th::holder_reg_almost_min}
		Let~$s\in(0,1)$. Let~$E$ be a~$(\Lambda,\gamma)$-almost minimizer for the $s$-Perimeter in $B_1$.
		
		Then, there exist~$\epsilon_0>0$ and~$\eta_0=\eta_0(\epsilon_0)>0$, depending only on~$n$,~$s$,~$\Lambda$, and~$\gamma$, such that if
		$$\partial E\cap B_1 \subseteq \{|x_n|\leq \eta_0\},$$
		then~$\partial E\cap B_{1/2}$ is a~$\cont^{1,\alpha}$-surface, for any~$\alpha\in(0,\epsilon_0)$.
	\end{theorem}
	
	We point out that Theorem~\ref{th::holder_reg_almost_min} extends\footnote{In particular, we have that $\epsilon_0=s$, when $\gamma=s$. Morever, whenever $\Lambda=0$, we recover the well-known analog for $s$-minimal surfaces (see \cite[Theorem 6.1]{MR2675483}).} \cite[Theorem 1.4]{MR5015001}
	to almost minimal sets without any further additional assumption,
	making the corresponding regularity theory complete and fully applicable in the setting of almost minimal sets.\medskip
	
	The strategy to prove our result goes as follows: we first show that almost minimal sets satisfy suitable \emph{uniform density estimates} (see Section~\ref{sec::UDE}), namely almost minimizers of $\Per_s$ cannot present thin cusps. Then, we employ an \emph{improvement of flatness} technique for almost minimal sets (see Theorem~\ref{th::flatness_improvement}) to deduce Theorem~\ref{th::holder_reg_almost_min}. The flatness improvement argument is implemented in Section~\ref{sec::improvement_of_flatness}. 
	
	At this stage, we notice that the estimates in Section~\ref{sec::improvement_of_flatness} can be improved for $\cont^{1,\alpha}$ sets. Therefore, repeating iteratively the argument above, we refine Theorem~\ref{th::holder_reg_almost_min} and obtain an explicit H\"older exponent.
	
	\begin{cor} \label{cor::optimal_reg}
		Let~$s\in(0,1)$. Let~$E$ be a~$(\Lambda,\gamma)$-almost minimizer for the $s$-Perimeter in $B_1$.
		
		Then, there exist~$\eta_0>0$, depending only on~$n$,~$s$,~$\Lambda$, and~$\gamma$, such that if
		$$\partial E\cap B_1 \subseteq \{|x_n|\leq \eta_0\},$$
		then~$\partial E\cap B_{1/2}$ is a~$\cont^{1,\alpha}$-surface, for any~$\alpha\in(0,\overline{\epsilon})$, where
		\begin{equation*}
			\overline{\epsilon} = \frac{-(n+s+2-2\gamma)+\left[ (n+s+2-2\gamma)^2+4\gamma(n+s-\gamma)\right]^{1/2}}{2(n+s-\gamma)}.
		\end{equation*}
	\end{cor}
	\begin{rem}
		We point out that, for $\gamma$ small enough, we have that
		\begin{equation*}
			\begin{split}
				&\lim_{s\to1^-} \overline{\epsilon} = \frac{-(n+3-2\gamma)+\left[ (n+3-2\gamma)^2+4\gamma(n+1-\gamma)\right]^{1/2}}{2(n+1-\gamma)} \\
				&\qquad\leq \frac{2\gamma(n+1-\gamma)}{2(n+3-2\gamma)(n+1-\gamma)} = \frac{\gamma}{n+3-2\gamma} \leq \alpha_{opt},
			\end{split}
		\end{equation*}
		where
		\begin{equation*}
			\alpha_{opt} := \frac{n\gamma}{n+1-\gamma}
		\end{equation*}
		is the optimal H\"older exponent for the classical Massari's Problem, as shown in~\cite{MR4594364}.
		
		On the other hand, we have that
		\begin{equation*}
			\lim_{\gamma\to s} \overline{\epsilon} = \frac{-(n-s+2)+\left[ (n-s+2)^2+4sn\right]^{1/2}}{2n} \leq \frac{2sn}{2n(n-s+2)} \leq s,
		\end{equation*}
		whenever $s$ is small enough, which is consistent with \cite[Theorem 1.4]{MR5015001}.
	\end{rem}

	The proofs of Theorem~\ref{th::holder_reg_almost_min} and Corollary~\ref{cor::optimal_reg} can be found in Section~\ref{sec::holder_almost_min}.
	
	\subsection{Hausdorff dimension of the singular set}
	
	The next step in analyzing the regularity of almost minimal sets is to estimate the Hausdorff dimension of the so-called singular set. To do so, we introduce the following definitions.
	
	\begin{definition} \label{def::sing_reg_points}
		Let~$E$ be a~$(\Lambda,\gamma)$-almost minimal set in~$B_1$, for some $\Lambda\geq0$ and $\gamma\in(0,s]$, and let $x_0\in B_1\cap\partial E$. Let also $\{\lambda_k\}_k$ be a sequence of positive real numbers such that $\lambda_k\to+\infty$, as $k\to+\infty$. 
		
		If the blow-up sequence $\{\lambda_k (E-x_0)\}_k$ converges to a $0$-homogeneous set $C$, we say that $C$ is a tangent cone to $E$ at $x_0$. Moreover, we say that $x_0$ is a regular point if $C$ is a half-space.
		
		We say that~$x$ is singular if it is not regular. We denote by~$\Sigma_E$ the set of singular points for~$E$. In particular, $\Sigma_E$ is a closed set in~$\R^n$.
	\end{definition}
	
	Adapting the result in Sections 3 and 4 of \cite{MR5015001} to Definition~\ref{def::almost_min}, we deduce an upper bound on the Hausdorff dimension of 
	$\Sigma_E$. The precise statement goes as follows: 
	
	\begin{theorem} \label{th::haus_dim_singular_almost}
		Let~$E$ be a~$(\Lambda,\gamma)$-almost minimal set in $B_1$. 
		Then, $\dim_\haus\ (\Sigma_E)\leq n-3$.
	\end{theorem}
	
	The idea to establish Theorem~\ref{th::haus_dim_singular_almost} relies on a suitable \emph{monotonicity formula} and a \emph{blow-up argument}, which allow a classification of tangent cones and ultimately yields the desired estimate of the Hausdorff dimension of~$\Sigma_E$. 
	
	In particular, let $E\subseteq\R^n$ be a~$(\Lambda,\gamma)$-almost minimal set. We consider the $(n+1)$-dimensional extension $\tilde{u}$ of $\chi_E-\chi_{E^c}$ constructed as in \cite{MR2354493} (see also \cite[Section~7]{MR2675483}) and we define the function
	\begin{equation}\label{eq::defphimono}
		\Phi_E(r) := r^{s-n}\int_{B_r^+} z^{1-s}|\nabla\widetilde{u}|^2\,dx\,dz + \frac{n-s}{\gamma}\hat{\Lambda}r^{\gamma},
	\end{equation}
	where~$ \omega_n$ is the Lebesgue measure of the~$n$-dimensional unit ball, and~$\hat{\Lambda}$ is defined in Remark~\ref{rem::hat_Lambda} down below.
	Then, we have the following:
	\begin{prop}[Monotonicity formula] \label{prop::Phi_monotone}
		If~$E$ is a~$(\Lambda,\gamma)$-almost minimal set in $B_1$ such that~$0\in\partial E$, then~$\Phi_E(r)$ is increasing for any~$r$ small enough. 
	\end{prop} 
	
	A proof of Theorem~\ref{th::haus_dim_singular_almost} is given along with the blow-up argument in Section~\ref{sec::proof_haus_dim_sing}. Moreover, since the arguments involved are variations of those presented in~\cite[Section~3]{MR5015001}, we postpone some details regarding the monotonicity formula to Appendix~\ref{sec::monotonicity}. 
	
	\begin{rem}
		We point out that the sharpness of Theorem~\ref{th::haus_dim_singular_almost} is still an open problem, even for $s$-minimal surfaces. Indeed, in the classical setting, it is well know that minimal surfaces in $\R^n$ are smooth up to dimension $n=7$, whereas the Simons cone, defined by
		$$\{(x,y)\in\R^4\times\R^4 \mbox{ s.t. }|x|^2=|y|^2\},$$ 
		provides an example of a singular minimal surface in $\R^8$.
	\end{rem}\medskip
	
	We now proceed with the proof of the main results.
	
	
	\section{Uniform density estimates} \label{sec::UDE}
	As claimed in the introduction, the first step to prove the regularity for almost minimal sets is to show the following density estimates:
	\begin{prop}[Uniform density estimates] \label{prop::UDE} 
		Let~$E$ be a~$(\Lambda,\gamma)$-almost minimal set in $B_1$.
		Then, there exist constants~$r_0>0$ and~$c_0\in(0,1)$, depending only on~$n$, $s$,~$\Lambda$ and~$\gamma$, such that, for any~$x_0\in(\partial E)\cap B_1$ and~$r\in(0,\min\{r_0,\mbox{dist}(x_0,\partial B_1)\})$,
		\begin{equation}\label{eq::unif_dens_estimates}
			\begin{split}
				&|E\cap B_r(x_0)|\geq c_0 r^n\\
				\mbox{and}\quad&|E^c\cap B_r(x_0)|\geq c_0 r^n.
			\end{split}
		\end{equation}
	\end{prop}
	
	As a consequence, we obtain that local convergence\footnote{
		As customary, we say that a sequence of sets $\{E_k\}_k$ converges to some set $E$ in $L^1_{\mbox{loc}}(\R^n)$, and we simply write $E_k\to E$ in $L^1_{\mbox{loc}}(\R^n)$, if $\chi_{E_k}$ converges to $\chi_E$ in the $L^1_{\mbox{loc}}(\R^n)$-topology.
	} of almost minimal sets implies uniform convergence.
	
	\begin{cor}[Uniform convergence for almost minimal boundaries] \label{cor::boundary_conv}
		Let~$\{\Lambda_k\}_k$ be a sequence in $[0,+\infty)$, and let $\{\gamma_k\}_k$ be a sequence in $[0,s)$. Let also~$\{E_k\}_k$ be a sequence of~$(\Lambda_k,\gamma_k)$-almost minimal sets in $B_1$. 
		
		Suppose that~$\Lambda_k\to\Lambda$, for some $\Lambda\in[0,+\infty)$,~$\gamma_k\to\gamma$, for some $\gamma\in(0,s]$, and~$E_k\to E$ in~$L^1_{\loc}(\R^n)$, for some set~$E$ of finite~$s$-perimeter, as~$k\to+\infty$.
		
		Then, $E$ is $(\Lambda,\gamma)$-almost minimal in $B_1$ and
		$$ \lim_{k\to+\infty} \mbox{d}_\haus{}(\partial E_k,\partial E)=0 .$$		
	\end{cor}
	
	Another consequence of Proposition~\ref{prop::UDE} is the so-called \emph{clean ball condition} for almost minimal sets.
	
	\begin{cor}[Clean ball condition for almost minimal sets] \label{cor::clean_ball}
		Let~$E$ be a $(\Lambda,\gamma)$-almost minimal set in~$B_1$.
		Let~$x\in\partial E$ and~$r>0$ such that~$B_r(x)\subseteq B_1$.
		
		Then, there exist a constant~$c>0$, depending only on~$n$, $s$,~$\Lambda$ and~$\gamma$, and points~$y_1\in E$ and~$y_2\in B_1\setminus E$ such that
		$$ B_{cr}(y_1) \subseteq E\cap B_r(x)\qquad {\mbox{and}} \qquad B_{cr}(y_2) \subseteq E^c\cap B_r(x).$$
	\end{cor}
	
	The arguments to show Proposition~\ref{prop::UDE}, Corollary~\ref{cor::boundary_conv}, and Corollary~\ref{cor::clean_ball} are a variation\footnote{The arguments presented in Section 2 and Appendix B of \cite{MR5015001} carry over mutatis mutandis for all $\gamma\in(0,s)$. Indeed, their proofs rely only on density estimates and not on the specific definition of almost minimality.} of the ones for Theorem~2.2, Corollary~2.3, and Corollary~2.4 in \cite{MR5015001}. For the sake of completeness, we present a proof of Proposition~\ref{prop::UDE} in Appendix~\ref{sec::proof_UDE}. For a detailed discussion on Corollaries~\ref{cor::boundary_conv} and~\ref{cor::clean_ball}, we refer to Section 2 and Appendix B of \cite{MR5015001}. In particular, we point out that the key to show Corollary~\ref{cor::boundary_conv} is that the conclusion of Proposition~\ref{prop::UDE} holds for every set $E_k$ with constants $c_0$ and $r_0$ independent of $k$.
	
	\subsection{Proof of uniform density estimates (Proposition~\ref{prop::UDE})} \label{sec::proof_UDE}
	
	Now, we present a proof of Proposition~\ref{prop::UDE}. The argument relies on the following auxiliary results.
	
	\begin{lemma} \label{lemma::keep_small_density}
		Let $\delta_0>0$.
		Let~$E$ be a~$(\Lambda,\gamma)$-almost minimal set in $B_{1}$, with $\Lambda\leq\delta_0$. If
		\begin{equation} \label{eq::small_density}
			a(1)\leq \delta_0, 
		\end{equation}
		then $a(1/2)\leq \delta_0$. 
	\end{lemma}
	
	\begin{proof}
		Let $r\in\left(\frac{1}{2},1\right)$, and define $A_r:=E\cap B_r$. Also, let~$\mu(r):=|A_r|$ and notice that, by the co-area formula, $\mu'(r)=\haus{n-1}(E\cap\partial B_r)$.
		
		Now, observe that, by Sobolev embeddings, we have that  
		\begin{equation} \label{eq::ude1}
			\norm{\chi_{A_r}}_{L^{\frac{n}{n-s}}(\R^n)} = |A_r|^{1-\frac{s}{n}}\leq C[\chi_{A_r}]_{H^{s/2}} = C \mathcal{L}_s(A_r,A_r^c),
		\end{equation}
		for some constant~$C>0$ depending only on~$n$ and ~$s$.
		
		Since
		\begin{equation*}
			\mathcal{L}_s(A_r,A_r^c) = \mathcal{L}_s(A_r,E\cap A_r^c)+\mathcal{L}_s(A_r,E^c),
		\end{equation*}
		it follows from the almost minimality of $E$ that
		\begin{equation} \label{eq::ude2}
			\begin{split}
				&\mathcal{L}_s(A_r,A_r^c)\leq 2\mathcal{L}_s(A_r,E\cap A_r^c) + \Lambda r^{n-s+\gamma}\\
				&\qquad\leq 2\mathcal{L}_s(A_r,B_r^c) + \Lambda r^{n-s+\gamma}.
			\end{split}
		\end{equation}
		From~\eqref{eq::ude1} and~\eqref{eq::ude2} together, we deduce that
		\begin{equation} \label{eq::ude3}
			\mu(r)^{1-\frac{s}{n}} \leq C\mathcal{L}_s(A_r,B_r^c) + C\Lambda r^{n-s+\gamma}.
		\end{equation}
		and we aim to estimate the first term on the right-hand side.
		
		To do so, observe that, by the co-area formula, we have 
		\begin{equation} \label{eq::ude4}
			\begin{split}
				&	\mathcal{L}_s(A_r,B_r^c)
				= \iint_{A_r\times B_r^c} \frac{dx\,dy}{|x-y|^{n+s}}
				\leq C\int_{A_r}\left(\int_{r-|x|}^{+\infty}\frac{dz}{z^{s+1}}\right)\,dx \\
				&\qquad\quad\leq C\int_{A_r} \frac{dx}{(r-|x|)^{s}}
				\leq C\int_0^r\frac{\mu'(\rho)}{(r-\rho)^{s}}\, d\rho ,
			\end{split}
		\end{equation}
		for some~$C>0$, depending on~$n$ and~$s$, and possibly changing from line to line.
		
		Now, plugging \eqref{eq::ude4} into \eqref{eq::ude3}, we find
		\begin{equation*}  
			\mu(r)^{1-\frac{s}{n}} \leq C\int_0^r\frac{\mu'(\rho)}{(r-\rho)^{s}}\, d\rho + C\Lambda r^{n-s+\gamma}.
		\end{equation*}
		Integrating la last inequality over $r\in\left(\frac{1}{2},1\right)$, we obtain
		\begin{equation} \label{eq::ude6}
			\begin{split}
				&a(1/2)^{1-\frac{s}{n}} \leq C\mu(1/2)^{1-\frac{s}{n}} \leq C\int_{1/2}^{1} \mu(r)^{1-\frac{s}{n}}\, dr\\
				&\qquad\leq C\int_{1/2}^{1} \left( \int_0^r\frac{\mu'(\rho)}{(r-\rho)^{s}}\, d\rho + \Lambda r^{n-s+\gamma} \right) dr \\
				&\qquad\leq C\int_{0}^{1/2}\int_{1/2}^1\frac{\mu'(\rho)}{(r-\rho)^{s}}\, dr\,d\rho + C\int_{1/2}^{1}\int_\rho^1\frac{\mu'(\rho)}{(r-\rho)^{s}}\, dr\,d\rho +C\Lambda \int_{1/2}^{1} r^{n-s+\gamma}\,dr \\
				&\qquad \leq C\int_{0}^{1/2}\mu'(\rho)(1-\rho)^{1-s}\, d\rho + C\int_{1/2}^{1}\mu'(\rho)(1-\rho)^{1-s}\, d\rho +C\Lambda\\
				&\qquad \leq C\mu(1) +C\Lambda \leq Ca(1) +C\Lambda,				
			\end{split} 
		\end{equation}
		up to renaming $C$.
		
		Using~\eqref{eq::small_density} and that $\Lambda \leq \delta_0$ in~\eqref{eq::ude6}, we conclude that
		\begin{equation*} 
			a(1/2) \leq \left(Ca(1) +C\Lambda\right)^{\frac{n}{n-s}} \leq (C\delta_0^{s/n})^{\frac{n}{n-s}} \delta_0 \leq \delta_0
		\end{equation*}
		whenever $\delta_0$ is small enough, showing the result.
	\end{proof}
	
	The other partial result needed for the proof of Proposition~\ref{prop::UDE} is a clean-out lemma.
	
	\begin{lemma} \label{lemma::empty_ball}
		Let $\delta_0$ be as in Lemma~\ref{lemma::keep_small_density}. Let~$E$ be a~$(\Lambda,\gamma)$-almost minimal set in $B_{1}$, with $\Lambda\leq\delta_0$. 
		Moreover, suppose that
		\begin{equation} \label{eq::small_density2}
			a(1)<\frac{\delta_0}{\omega_n2^n}.
		\end{equation}
		Then, $|E\cap B_{1/2}|=0$.
	\end{lemma}
	
	\begin{proof}
		Arguing by contradiction, suppose that $|E\cap B_{1/2}|>0$. Therefore, there exists a Lebesgue point $x_0\in E\cap B_1$ such that 
		\begin{equation*}
			\lim_{\rho\to0} a(\rho) = 1.
		\end{equation*}
		
		On the other hand, by~\eqref{eq::small_density2}, we also have that 
		\begin{equation} \label{eq::recentered_ball_density}
			\frac{|E\cap B_{1/2}(x_0)|}{|B_{1/2}|} \leq 2^n|E\cap B_{1}| \leq \delta_0.
		\end{equation}
		
		Without loss of generality, we assume that $x_0=0$, so that~\eqref{eq::recentered_ball_density} reads $a(1/2)\leq \delta_0$. Now, thanks to Lemma~\ref{lemma::keep_small_density}, we also have $a(1/4)\leq \delta_0$.
		Applying inductively Lemma~\ref{lemma::keep_small_density}, we infer that, for every $k\in\N$,
		\begin{equation*}
			a(2^{-k-1}) \leq \left(Ca(2^{-k})+C\Lambda\right)^{\frac{n}{n-s}} \leq \left( C\delta_0^{s/n}\right)^{\frac{n}{n-s}}\delta_0 \leq \delta_0.
		\end{equation*}
		
		Therefore,
		\begin{equation*}
			1 = \lim_{k\to+\infty} a(2^{-k-1}) \leq \delta_0,
		\end{equation*}
		which is a contradiction and concludes the proof.
	\end{proof}
	
	\begin{proof}[Proof of Proposition~\ref{prop::UDE}]
		Let $r_0>0$ such that $\Lambda r_0^\gamma\leq \delta_0$,~where $\delta_0$ is as in Lemma~\ref{lemma::keep_small_density}. 
		
		Suppose that
		\begin{equation} \label{eq::some_mass}
			|E \cap B_\rho|>0,
		\end{equation}
		and let~ $x_0\in( \partial E)\cap B_\rho$, for some $\rho\in(0,r_0)$. Up to a dilations and translation, we suppose that~$x_0$ coincides with the origin. 
		
		Let us recall that, by definition, if~$E$ is a~$(\Lambda,\gamma)$-almost minimal set in $B_1$, then
		\begin{equation*} \label{eq::almost_min2}
			\Per_s(E, B_\rho) - \Per_s(F,B_\rho) = \Per_s(E, B_1) - \Per_s(F,B_1) \leq \Lambda \rho^{n-s+\gamma},
		\end{equation*}
		for every~$F$ such that~$ F\Delta E \subseteq B_\rho \Subset B_1$.
		Observe that, rescaling by~$\rho$, we obtain that
		\begin{equation} \label{eq::still}
			\Per_s(\rho^{-1}E,B_1) \leq \Per_s(\rho^{-1}F,B_1) + \Lambda\rho^\gamma.       
		\end{equation}
		Thus, without loss of generality, we assume $\rho=1$ and that $E$ is $(\Lambda,\gamma)$-quasi minimal in $B_{1}$, with $\Lambda\leq\delta_0$. Moreover,~\eqref{eq::some_mass} reads
		\begin{equation} \label{eq::some_mass2}
			|E \cap B_1|>0.
		\end{equation}
		\smallskip
		We claim that 
		\begin{equation} \label{eq::ude_claim}
			|E\cap B_1|\geq \delta_0.
		\end{equation}
		
		Arguing by contradiction, suppose that
		\begin{equation*} \label{eq::ude_small}
			|E\cap B_1|\leq \delta_0.
		\end{equation*}
		
		Thus, by mean of Lemma~\ref{lemma::empty_ball}, we infer $|E\cap B_1|=0$. However, this contradicts~\eqref{eq::some_mass2}, proving~\eqref{eq::ude_claim}. Rescaling back by $\rho$, we obtain the first part of~\eqref{eq::unif_dens_estimates}.
		\smallskip
		
		In addition, if~$E$ is a quasi minimal set, then also~$E^c$ is quasi minimal . 
		Thus, the argument above with $E^c$ in place of $E$ leads to the second part of~\eqref{eq::unif_dens_estimates}, concluding the proof.
	\end{proof}
	
	\begin{rem}
		We point out that, if~$\{\Lambda_k\}_k$,~$\{\gamma_k\}_k$, and~$\{E_k\}_k$ are sequences satisfying the assumptions of Corollary~\ref{cor::boundary_conv}, then there exist $\Lambda\geq0$ and~$\gamma\in(0,s]$ such that $\Lambda_k\to\Lambda$ and $\gamma_k\to\gamma$, as $k\to+\infty$. 
		In particular, there exist $\Lambda_0\geq0$ and~$\gamma_0\in[0,s)$ such that $\Lambda_k\leq\Lambda_0$ and $\gamma_k\leq\gamma_0$, for every $k$ large enough. Therefore, the conclusion of Proposition~\ref{prop::UDE} holds for every set $E_k$ with constants $c_0$ and $r_0$ depending only on~$n$,~$s$,~$\Lambda_0$, and~$\gamma_0$. In particular, $c_0$ and $r_0$ are independent of $k$.
	\end{rem}
	
	
	\section{Improvement of flatness}\label{sec::improvement_of_flatness}
	
	This section is devoted to implement a flatness improvement technique useful to deduce Theorem~\ref{th::holder_reg_almost_min}.
	To this purpose, we now introduce the notion of flatness for a set. For later convenience, we define, for any $x_0\in\R^n$, $\nu\in \S^{n-1}$, and $r\in\R$, the cylinder
	\begin{equation*}
		C_r^\nu(x_0) := \{(x',x_\nu)\in\R^n \mbox{ s.t. } |(x'-x_0')|<r\} = B_r^{n-1}(x_0)\times\nu\R.
	\end{equation*}
	When $\nu=e_n$, we will denote the vertical cylinder of radius $r$ centered at $x_0$ simply with $C_r(x_0)$.
	Also, when the cylinders are centered at the origin, we use the short notations~$C_r^\nu:=C_r^\nu(0)$
	and~$C_r:=C_r(0)$.
	
	\begin{definition}[{\cite[Definition~2.13]{MR5015001}}] \label{def::flatness}
		Let~$h$, $r>0$ and~$x_0\in\partial E$. We say that a set~$E$ is~$\left(\frac{h}{r}\right)$-flat in $C_r^\nu(x_0)$ if 
		$$ \partial E\cap C_r^\nu(x_0)\subseteq\big\{|(x_n-x_0)\cdot\nu|<h\big\}.$$
	\end{definition}
	In other words, we say that a set is~$\left(\frac{h}{r}\right)$-flat in the cylinder $C_r^\nu(x_0)$ if it can be trapped in a slab of width~$2h$ in direction $\nu$. 
	\medskip
	
	The idea of the improvement of flatness technique is to show that if an almost minimal set has controlled flatness in each ball of radius $2^{-k}$, for all $k=0,\dots,k_0$, then its flatness improves in every dyadic ball with smaller radius.
	
	\begin{theorem}[Improvement of flatness] \label{th::flatness_improvement}	
		Let~$\alpha>0$. 
		Let~$E$ be~a~$(\Lambda,\gamma)$-almost minimizer for~$\Per_s$ in~$B_{1}$, and let~$x_0\in\partial E\cap B_{1/2}$. 
		
		Then, there exist~$k_0\in\N$ depending on $n$,~$s$,~$\Lambda$,~$\gamma$, and~$\alpha$,
		such that, if
		\begin{equation}  \label{eq::tail_control}
			\partial E\cap C^{\nu_j}_{2^{-j}}\subseteq\big\{|x\cdot\nu_j|\leq 2^{-j(1+\alpha)}\big\},\qquad {\mbox{for all }} j=0,\dots,k \mbox{, for some }k\geq k_0,
		\end{equation}
		where~$\{\nu_j\}_{j=0}^k$ is a family of unit vectors with~$\nu_0=e_n$, then there exist unit vectors~$\{\nu_j\}_{j\geq k+1}$ such that 
		\begin{equation*}
			\partial E\cap C^{\nu_j}_{2^{-j}}\subseteq\big\{|x\cdot\nu_j|\leq2^{-j(1+\alpha)}\big\},\qquad {\mbox{for all }} j\geq k+1.
		\end{equation*}
	\end{theorem}

	Theorem~\ref{th::flatness_improvement} follows easily by rescaling from the following:
	
	\begin{prop}[Rescaled improvement of flatness] \label{prop::improv_flat_rescaled}
		Let $\Lambda\geq0$ and~$\gamma\in(0,s]$, and let~$\alpha>0$. 
		
		Then, there exists~$k_0\in\N$, depending only on~$n$,~$s$,~$\Lambda$,~$\gamma$, and~$\alpha$, such that the following holds.
		Let $k\geq k_0$ and let~$E$ be~$(2^{-k\gamma}\Lambda,\gamma)$-almost minimal in~$B_{2^k}$ with~$0\in\partial E$. Assume that
		\begin{equation} \label{eq::rescaled_flat_improve_hp}
			\partial E\cap C_{2^j}\subseteq \big\{|x\cdot\nu_j|\leq 2^{-k\alpha}2^{j(1+\alpha)}\big\},\qquad {\mbox{for all }} j=0,\dots,k.
		\end{equation}
		
		Then there exists a unit normal~$\nu_{-1}$ such that 
		\begin{equation} \label{eq::flatness_1/2}
			\partial E\cap C_{1/2}\subseteq \big\{|x\cdot\nu_{-1}|\leq 2^{-k\alpha-\alpha-1}\big\}.
		\end{equation}
	\end{prop}
	\medskip
	
	The strategy to prove Proposition~\ref{prop::improv_flat_rescaled} relies on a contradiction argument. Let $\Lambda\geq0$ and~$\gamma\in(0,s]$, and let~$\alpha>0$.  
	
	We consider a sequence of~$(2^{-k\gamma}\Lambda,\gamma)$-almost minimal sets~$\{E_k\}_k$. For every~$k\in\N$, we define~$a_k:=2^{-k\alpha}$, and we assume that there exist~$j_k\in\N$ and a collection of unit vector $\{\nu_j^k\}_{j=0}^{j_k}$ such that
	\begin{equation}\label{eq::flatness_Ek2}
		\partial E_k\cap C_{2^j} \subseteq \big\{|x\cdot\nu_j^k|\leq a_k2^{j(1+\alpha)}\big\},\quad{\mbox{for all }}j\in\{ 0,\dots, j_k\} ,
	\end{equation} 
	with~$\nu_0^k=e_n$ for every~$k$.
	
	By contradiction, suppose that~
	\begin{equation} \label{eq::claim_contradiction}
		\mbox{$\partial E_k$ cannot fit in any cylinder of height~$2^{-(k+1)\alpha}$ in~$C_{1/2}$. }
	\end{equation}
	Our goal is to show that there exists a subsequence~$\{E_{k_l}\}_l$ uniformly convergent to a subgraph of some linear function~$u$, leading to the desired contradiction.
	
	For later convenience, let us denote by~$E_k^*$ the vertical dilation of~$E_k$ of a factor~$a_k$, i.e.
	$$ E_k^*:=\big\{(x',x_n)\;{\mbox{ s.t. }}\; (x',a_kx_n)\in E_k\big\}.$$
	
	
	
	\subsection{Euler-Lagrange inequalities}\label{sec::EL} 
	
	We remark that, as anticipated in the introduction, the main motivation for introducing almost minimizers (even in the classical case) is for being able to deal with object that are not proper minimizers of the energy functional under investigation. Therefore, such almost minimizers do not necessarily satisfy any Euler-Lagrange equation. On the other hand, we show here that
	one is still able to infer some {\em Euler-Lagrange inequalities at smaller scales}.
	
	In our setting, almost minimizers for the $s$-perimeter (in the sense of Definition~\ref{def::almost_min}) do not satisfy the mean curvature equation (compare with \cite[Theorem 5.1]{MR2675483}). However, if $E$ is an almost minimal set with a ball touching from below at some point, then the non-local mean curvature of $E$ is bounded outside a small ball centered at the touching point. The precise statement goes as follows:
	
	\begin{theorem} \label{th::EL}
		Let $E$ be a $(\Lambda,\gamma)$-almost minimal set in $B_1$. Assume that there exists a ball of radius~$2r$ touching $E$ at some $x_0\in\partial E$ from below.
		
		Then, there exists $\epsilon=\epsilon(n,s,\gamma)>0$ such that 
		\begin{equation}\label{cor::EL2}
			H^{r^{1+\epsilon/2}}_s[E](x_0) := \int_{B_{r^{1+\epsilon/2}}^c(x_0)} \frac{\chi_E(y)-\chi_{E^c}(y)}{|x_r-y|^{n+s}}dy \leq cr^{-s+\epsilon\frac{1-s}{2}}, 
		\end{equation}
		for some constant $c>0$ depending only on $n$, $s$,~$\Lambda$, and~$\gamma$.
	\end{theorem}
	
	Without loss of generality, let us suppose that $x_0=0\in\partial E$, so that, up to rotations, the ball $B_{2r}(-2re_n)$ touches $\partial E$ at $0$ from below. 
	Our proof relies on the two following partial results.
	
	\begin{prop} \label{prop::EL_partial1}
		Let $E$ be a $(\Lambda,\gamma)$-almost minimal set in $B_1$. Assume that the ball $B_{2r}(-2re_n)$ touches $E$ at $0\in\partial E$ from below.
		
		Then, there exists $\epsilon=\epsilon(n,s,\gamma)>0$ and a point $x_r\in E^c\cap B_{r^{1+\epsilon}}$ such that 
		\begin{equation}\label{eq::EL}
			H_s[E](x_r) := \int_{\R^n} \frac{\chi_E(y)-\chi_{E^c}(y)}{|x_r-y|^{n+s}}dy \leq cr^{-s+\left(\gamma-\epsilon(n+s-\gamma)\right)},
		\end{equation}
		for some constant $c>0$ depending only on $n$, $s$, and $\Lambda$.
	\end{prop}
	
	\begin{prop} \label{prop::EL_partial2}
		Let $E$ be a $(\Lambda,\gamma)$-almost minimal set in $B_1$. Assume that the ball $B_{2r}(-2re_n)$ touches $E$ at $0\in\partial E$ from below.
		
		Then, there exists $\epsilon=\epsilon(n,s,\gamma)>0$ and a point $x_r\in E^c\cap B_{r^{1+\epsilon}}$ such that 
		\begin{equation} \label{eq::EL_partial2}
			\int_{B_{r^{1+\epsilon/2}}^c} \frac{\chi_E(y)-\chi_{E^c}(y)}{|x_r-y|^{n+s}}dy \leq cr^{-s+\epsilon\frac{1-s}{2}}.
		\end{equation}
		for some constant $c>0$ depending only on $n$,~$s$,~$\Lambda$, and~$\gamma$.
	\end{prop}
	
	\begin{proof}[Proof of Proposition~\ref{prop::EL_partial1}]
		Let us take $\epsilon>0$, to be precised later. We define $\rho:=r^{1+\epsilon}$ and $d:=r^{1+2\epsilon}$. Let us also consider the smaller ball $B_{r}(-re_n)$ and the sets
		\begin{equation} \label{eq::def_bubble}
			A_t := E^c\cap B_{r}((t-r)e_n),\quad \mbox{for every }t\in[0,d].
		\end{equation}
		
		We consider the set $A_d$ and we observe that, by construction, $E^c\cap B_d\subseteq A_d$. Thus, from the almost minimality of $E$ and the uniform density estimates in Proposition~\ref{prop::UDE}, we infer that 
		\begin{equation} \label{eq::EL_upper_bound1}
			\Per_s(E,B_1)-\Per_s(E\cup A_d,B_1) \leq \Lambda \rho^{n-s+\gamma} = \Lambda r^{(n-s+\gamma)(1+\epsilon)}.
		\end{equation}
			
			On the other hand, by calibration for the $s$-perimeter (see for instance \cite[Proposition 5.2]{MR3401008}), we have that
			\begin{equation} \label{eq::calib}
				\Per_s(E,B_1)-\Per_s(E\cup A_d,B_1) = \int_{A_d} H_s[E\cup A_{t_x}](x) dx,
			\end{equation}
			where $t_x\in[0,d]$ is such that $x\in \partial(E\cup A_{t_x})$.
			
			Arguing by contradiction, let us suppose that, for (almost) every $x\in A_d$, we have 
			$$H_s[E\cup A_{t_x}](x)> \frac{\Lambda}{c_0} r^{-s+\left(\gamma - \epsilon(n+s-\gamma)\right)}, $$ 
			where the constant $c_0>0$ is as in Proposition~\ref{prop::UDE}.
			
			Then, we deduce from~\eqref{eq::calib} and Proposition~\ref{prop::UDE} that
			\begin{equation*}
				\begin{split}
					&\Per_s(E,B_1)-\Per_s(E\cup A_d,B_1) > \frac{\Lambda}{c_0} r^{-s+\left(\gamma - \epsilon(n+s-\gamma)\right)}|A_d| \\
					&\qquad \geq \frac{\Lambda}{c_0} r^{-s+\left(\gamma - \epsilon(n+s-\gamma)\right)} |E^c \cap B_d| \geq \Lambda r^{-s+\left(\gamma - \epsilon(n+s-\gamma)\right)} d^n =  \Lambda r^{(n-s+\gamma)(1+\epsilon)}.
				\end{split}
			\end{equation*}
			
			Plugging this into \eqref{eq::EL_upper_bound1}, we obtain a contradiction. Therefore, there exist $t_r\in[0,d]$ and $x_r\in \partial(E\cap A_{t_r} ) \subset E^c\cap A_d$ such that
			\begin{equation} \label{eq::curv_bound}
				H_s[E\cup A_{t_r}](x_r) \leq  \frac{\Lambda}{c_0} r^{-s+\left(\gamma-\epsilon(n+s-\gamma)\right)}.
			\end{equation}
			Moreover, choosing $\epsilon\in\left(0,\frac{\gamma}{n+s-\gamma}\right)$, we have that $\gamma-\epsilon(n+s-\gamma)>0$.
		\end{proof}
		
		\begin{proof}[Proof of Proposition~\ref{prop::EL_partial2}]
			Let $x_r$ and $t_r$ be as in the proof of Proposition~\ref{prop::EL_partial1}. Let us also define $\xi:=r^{1+\epsilon/2}$. Then, we consider the quantity
			
			\begin{equation} \label{eq::weak_curv}
				\int_{B_{\xi}^c} \frac{\chi_E(y)-\chi_{E^c}(y)}{|x_r-y|^{n+s}}dy = H_s[E\cup A_{t_r}](x_r) - \int_{B_{\xi}} \frac{\chi_{E\cup A_{t_r}}(y)-\chi_{({E\cup A_{t_r}})^c}(y)}{|x_r-y|^{n+s}}dy 
			.\end{equation}
			
			Observe that the set $E\cup A_{t_r}$ is touched at $x_r$ from below by a parabula of opening $-r^{-1}$, and $x_r\in B_{\xi/2}$, whenever $r$ is small enough. Let us denote with $\phi$ such a parabula, and $P:=\{y_n<\phi(y')\}$. Then, recalling \cite[Appendix~B]{MR3926519} (see also \cite{MR3107529, MR3331523}), we have that
			\begin{equation} \label{eq::EL_upper_bound2}
				\begin{split}
					&\int_{B_{\xi}} \frac{\chi_{E\cup A_{t_r}}(y)-\chi_{({E\cup A_{t_r}})^c}(y)}{|x_r-y|^{n+s}}dy \geq \int_{B_{\xi}} \frac{\chi_{P}(y)-\chi_{P^c}(y)}{|x_r-y|^{n+s}}dy\\
					&\qquad\qquad =-\int_{B'_{\xi}} \int_{\frac{\nabla\phi(x_r)\cdot(y'-x_r')}{|x_r'-y'|}}^{\frac{\phi(y')-\phi(x_r')}{|x_r'-y'|}} \frac{dt}{(1+t^2)^{n+s}} \frac{dy'}{|x_r'-y'|^{n+s-1}} \\ 
					&\qquad\qquad \geq - \int_{B'_{\xi}} \frac{\norm{\phi}_{\cont^2(B'_{\xi})}|x_r'-y'|^2}{|x_r'-y'|} \frac{dy'}{|x_r'-y'|^{n+s-1}} \\
					&\qquad\qquad =	-\norm{\phi}_{\cont^2(B'_{\xi})} \int_{B'_{\xi}} \frac{dy'}{|x_r'-y'|^{n+s-2}} \\
					&\qquad\qquad \geq -\frac{c}{r}\int_{0}^{\frac{3}{2}\xi} \sigma^{-s}d\sigma = -cr^{-1}\xi^{1-s} = -cr^{-s+\frac{\epsilon}{2}(1-s)}.
				\end{split}
			\end{equation}
			for some positive constant $c$ depending only on $n$ and $s$.
			
			Thus, using also Proposition~\ref{prop::EL_partial1}, we infer that
			\begin{equation*}
				\int_{B_{\xi}^c} \frac{\chi_E(y)-\chi_{E^c}(y)}{|x_r-y|^{n+s}}dy \leq cr^{-s+\left( \gamma-\epsilon(n+s-\gamma)\right)} + cr^{-s+\epsilon\frac{1-s}{2}} , 
			\end{equation*}
			up to renaming $c$, now depending also on $\Lambda$ and $\gamma$.
			
			Taking $\epsilon=\frac{2\gamma}{2n+1+s-2\gamma}\in\left(0,\frac{\gamma}{n+s-\gamma}\right)$, we obtain
			\begin{equation} \label{eq::epsilon_choice}
				\gamma-\epsilon(n+s-\gamma) = \frac{\epsilon}{2}(1-s),
			\end{equation}
			and hence
			\begin{equation*}
				\int_{B_{\xi}^c} \frac{\chi_E(y)-\chi_{E^c}(y)}{|x_r-y|^{n+s}}dy \leq  cr^{-s+\epsilon\frac{1-s}{2}} . 
			\end{equation*}
		\end{proof}
		
		\begin{proof}[Proof of Theorem~\ref{th::EL}]
			Up to translations, we assume that $x_0=0\in\partial E$. Moreover, we define
			$$ \xi := r^{1+\frac{\epsilon}{2}} = \rho^{1-\frac{\epsilon}{2(1+\epsilon)}},$$
			so that $\rho/\xi=r^{\epsilon/2}\to0$, as $r\to0$. Let also $x_r$ be as in Proposition~\ref{prop::EL_partial1}.
			
			Then, for any $y\in B_\xi^c$, we have that
			\begin{equation*}
				1-r^{\epsilon/2}\leq \frac{|x_r-y|}{|y|} \leq 1+r^{\epsilon/2}. 
			\end{equation*}
			
			Therefore, for a small $r$, we obtain that
			\begin{equation*}
				\begin{split}
					&H^{\xi}_s[E](0) = \int_{B_{\xi}^c} \frac{\chi_E(y)-\chi_{E^c}(y)}{|y|^{n+s}}dy =  \int_{B_{\xi}^c} \frac{\chi_E(y)-\chi_{E^c}(y)}{|x_r-y|^{n+s}}\left(\frac{|x_r-y|}{|y|}\right)^{n+s}dy\\
					&\qquad\qquad\leq \int_{B_{\xi}^c} \frac{\chi_E(y)-\chi_{E^c}(y)}{|x_r-y|^{n+s}}dy + (cr^{\epsilon/2}+O(r^{\epsilon}))\int_{B_\xi^c}\frac{dy}{|x_r-y|^{n+s}}\\
					&\qquad\qquad \leq cr^{-s+\epsilon\frac{1-s}{2}} +cr^{\epsilon/2}\int_{\xi/2}^{+\infty}\sigma^{-1-s}d\sigma \leq cr^{-s+\epsilon\frac{1-s}{2}},
				\end{split}
			\end{equation*}
			for some positive constant $c=c(n,s,\Lambda,\gamma)$, possibly changing from line to line.
		\end{proof}
		
		Whenever an almost minimal set $E$ is touched from above, Theorem~\ref{th::EL} applied to $E^c$ produces the analogous lower bound for the fractional mean curvature at the touching point.
		
		\begin{cor} \label{cor::reverseELeq}
			Let $E$ be a $(\Lambda,\gamma)$-almost minimal set in $B_1$. Assume that there exists a ball of radius $2r$ touching $E$ at some $x_0\in\partial E$ from above.
			
			Then, there exists $\epsilon=\epsilon(n,s,\gamma)>0$ such that 
			\begin{equation*} 
				H^{r^{1+\epsilon/2}}_s[E](x_0) \geq -cr^{-s+\epsilon\frac{1-s}{2}}, 
			\end{equation*}
			for some constant $c>0$ depending only on $n$, $s$,~$\Lambda$, and~$\gamma$.
		\end{cor}
		
		
		\subsection{Harnack-type inequalities} \label{sec::harnack}
		
		In this section, we present a Harnack-type result for almost-minimizers. An iterative application of this result is the main ingredient in establishing the convergence of the sequence $\{E_k\}_k$ defined in~\eqref{eq::flatness_Ek2}.   
		
		\begin{theorem}[Harnack-type Inequalities] \label{th::harnack_ineq}
			Let $\Lambda\geq0$ and~$\gamma\in(0,s]$, and let~$E\subseteq\R^n$. Let also $\epsilon>0$ be as in Theorem~\ref{th::EL}.
			
			Then, for any~$\alpha\in\left(0,\frac{\epsilon}{2+\epsilon}\right)$, there exist positive constants $\delta_0$ and~$k_0$, depending only on~$n$,~$s$,~$\Lambda$,~$\gamma$, and~$\alpha$, such that, setting~$a:=2^{-\alpha k}$ for some~$k\geq k_0$, if~$E$ is~$(2^{-k\gamma}\Lambda,\gamma)$-almost minimal for~$\Per_s$ in~$B_{2^k}$, and there exists a collection $\{\nu_j\}_j\subseteq \S^{n-1}$ such that 
			\begin{equation}  \label{eq::flatness_harnack}
				\partial E\cap C_{2^j}\subseteq\{|x\cdot\nu_j|\leq a2^{j(1+\alpha)}\},\qquad {\mbox{for all }} j=0,\dots,k ,
			\end{equation}
			with~$\nu_0=e_n$, then we have the inclusions
			\begin{eqnarray}
				\label{eq::harnack_ineq1}\text{either} \quad&\partial E\cap C_{\delta_0} \subseteq \big\{x_n\leq a(1-\delta_0^2)\big\}\\
				\label{eq::harnack_ineq2}\text{or} \quad &\partial E\cap C_{\delta_0} \subseteq \big\{x_n\geq a(\delta_0^2-1)\big\}.
			\end{eqnarray}
		\end{theorem}
		
		We point out that, even if we have a control on the fractional mean curvature only outside a small neighborhood of a touching point, we are still able to deduce Theorem~\ref{th::harnack_ineq} from Theorem~\ref{th::EL}. The idea of the proof goes along the lines of the analogous result in \cite[Lemma 6.9]{MR2675483} (see also \cite[Theorem 2.14]{MR5015001}). Theorem~\ref{th::harnack_ineq} is a direct consequence of the following:
		
		\begin{lemma}[Partial Harnack-type Inequality]\label{lemma::partial_harnack}
			Let $\Lambda\geq0$ and~$\gamma\in(0,s]$, and let~$E\subseteq\R^n$. Let also $\epsilon>0$ be as in Theorem~\ref{th::EL}.
			
			Then, for any~$\alpha\in\left(0,\frac{\epsilon}{2+\epsilon}\right)$, there exist positive constants $\delta_0$ and~$k_0$, depending only on~$n$,~$s$,~$\Lambda$,~$\gamma$, and~$\alpha$, such that, setting~$a:=2^{-\alpha k}$ for some~$k\geq k_0$, if~$E$ is a~$(2^{-k\gamma}\Lambda,\gamma)$-almost minimal set in~$B_{2^k}$ such that
			\begin{equation}  \label{eq::flatness_harnack_lemma}
				\partial E\cap C^{\nu_j}_{2^j}\subseteq\big\{|x\cdot\nu_j|\leq a2^{j(1+\alpha)}\big\},\qquad{\mbox{for all }} j=0,\dots,k ,
			\end{equation}
			with~$\nu_0=e_n$, and
			\begin{equation} \label{eq::half_measure_cylinder}
				|E\cap C_{\delta_0,a}| \geq \frac{1}{2}|C_{\delta_0,a}| = \omega_{n-1}\delta_0^{n-1}a,
			\end{equation}where
			$$C_{\delta_0,a}:=B^{n-1}_{\delta_0}\times(-a,a),$$
			then we have that
			\begin{eqnarray} 
				\label{eq::inclusion_harnack1}\mbox{either}\quad&C_{\delta_0}\cap\big\{x_n<-a(1-\delta_0^2)\big\}\subseteq E\\
				\label{eq::inclusion_harnack2}\mbox{or}\quad&C_{\delta_0}\cap\big\{x_n>a(1-\delta_0^2)\big\}\subseteq E.
			\end{eqnarray}
		\end{lemma}  
		
		\begin{proof}
			Let~$k_0=k_0(\delta_0)$ be such that~$a\leq 2^{-k_0\alpha}<\delta_0$.
			
			From~\eqref{eq::flatness_harnack_lemma}, we have $\partial E\cap C_1 \subseteq \{|x_n|<a\}$. We claim that if~$\{x_n<-a\}\cap C_1 \subseteq E$, then~\eqref{eq::inclusion_harnack1} holds true.
			
			Suppose by contradiction that this does not hold, namely there exists a portion of~$\partial E \cap C_{\delta_0}$ trapped in the strip~$\{-a<x_n<-(1-\delta_0^2)a\}$.
			
			Then, we find a parabola of opening~$-\frac{a}{2}$ touching~$\partial E$ from below in~$C_{\delta_0}$.Let~$y\in C_{\delta_0}$ be the tangent point, so that 
			$$ -a\leq y_n \leq -(1-\delta_0^2)a.$$
			
			Now, we estimate  the fractional mean curvature of~$E$ at~$y$. Taking~$\delta_0<1/2$, we bound the far-away contributions as
			\begin{equation} \label{eq::harnack_proof1} 
				\begin{split}
					&	\left| \int_{B_{1/2}^c(y)}  \frac{\chi_E(x)-\chi_{E^c}(x)}{|x-y|^{n+s}}\,dx  \right|
					\\	\leq& \sum_{j=0}^{k} 	\left| \int_{B_{2^j}(y)\setminus B_{2^{j-1}}(y)}  \frac{\chi_E(x)-\chi_{E^c}(x)}{|x-y|^{n+s}}\,dx  \right| 
					+ \left| \int_{B_{2^k}^c(y)}  \frac{\chi_E(x)-\chi_{E^c}(x)}{|x-y|^{n+s}}\,dx  \right| \\
					\leq& \sum_{j=0}^{k} 	\left| \int_{B_{2^j}(y)\setminus B_{2^{j-1}}(y)}  \frac{\chi_E(x)-\chi_{E^c}(x)}{|x-y|^{n+s}}\,dx  \right| 
					+ \omega_n\int_{2^k}^{+\infty}  \frac{dr}{r^{1+s}} . 
				\end{split}
			\end{equation}
			
			Thanks to the flatness assumptions~\eqref{eq::flatness_harnack_lemma}, we have 
			\begin{equation*}
				\partial E\cap B_{2^j}(y) \subseteq \partial E\cap C^{\nu_{j+1}}_{2^{j+1}} \subseteq \big\{|x\cdot\nu_{j+1}| \leq a 2^{(j+1)(\alpha+1)}\big\} .
			\end{equation*}
			Moreover,~$y\in C_{\delta_0}$ yields
			\begin{equation*}
				|y\cdot\nu_{j+1}|  \leq a 2^{(j+1)(\alpha+1)}.
			\end{equation*}
			Combining the last two displays, we obtain
			\begin{equation} \label{eq::flatness_harnack2}
				\partial E\cap B_{2^j}(y) \subseteq \big\{|(x-y)\cdot\nu_{j+1}| \leq 2 a 2^{(j+1)(\alpha+1)}\big\} .
			\end{equation}
			
			Let us define 
			$$ D := \left(B_{2^j}(y)\setminus B_{2^{j-1}}(y)\right) \setminus \big\{ |(x-y)\cdot \nu_{j+1}| \leq 2a2^{(j+1)(\alpha+1)} \big\} .$$
			By symmetry, we infer that
			\begin{equation*}
				\int_{D} \frac{\chi_E(x)-\chi_{E^c}(x)}{|x-y|^{n+s}} dx = 0,
			\end{equation*}
			and hence
			\begin{equation} \label{eq::harnack_proof2}
				\begin{split}
					&	\left| \int_{B_{2^j}(y)\setminus B_{2^{j-1}}(y)}  \frac{\chi_E(x)-\chi_{E^c}(x)}{|x-y|^{n+s}}\,dx  \right| 
					\leq \int_{B_{2^j}(y)\setminus B_{2^{j-1}}(y)}  \frac{\chi_{\left\{|(x-y)\cdot\nu_{j+1}|\leq2 a2^{(j+1)(\alpha+1)}\right\}}(x)}{|x-y|^{n+s}}\,dx \\
					&\qquad\qquad \leq \omega_n \int_{2^{j-1}}^{2^j} \frac{4 a 2^{(j+1)(\alpha+1)}r^{n-2}}{r^{n+s}} \,dr 
					= C a \int_{2^{j-1}}^{2^j} \frac{dr}{r^{1+s-\alpha}} ,
				\end{split}
			\end{equation}
			for some~$C>0$, depending only on~$n$ and~$\alpha$.
			
			Therefore, from~\eqref{eq::harnack_proof1} and~\eqref{eq::harnack_proof2}, 
			we deduce that
			\begin{equation} \label{eq::harnack_proof3}
				\begin{split}
					&\left| \int_{B_{1/2}^c(y)}  \frac{\chi_E(x)-\chi_{E^c}(x)}{|x-y|^{n+s}}\,dx  \right|
					\leq \sum_{j=0}^{k} 	C a \int_{2^{j-1}}^{2^j} \frac{dr}{r^{1+s-\alpha}} + \omega_n\int_{2^k}^{+\infty}  \frac{dr}{r^{1+s}}    \\
					&\qquad\leq C a \int_{1/2}^{+\infty} \frac{dr}{r^{1+s-\alpha}}+ \frac{\omega_n}{s}2^{-sk} 
					\leq c_1 \big(a+2^{-\alpha k}\big) \leq c_1 a,
				\end{split}
			\end{equation}
			for some~$c_1>0$ depending only on~$n$, $s$, and~$\alpha$.
			
			Now, let us denote by~$P$ the subgraph of the tangent parabola~$x_n=y_n-\frac{a}{2}|x'-y'|^2$. Let us also consider $\rho:=2^{-\frac{k}{2}(\epsilon-\epsilon\alpha-2\alpha)}$, and notice that $\epsilon-2\alpha-\epsilon\alpha>0$ by assumption. Then,
			\begin{equation} \label{eq::harnack_proof4}
				\begin{split}
					& \int_{B_{1/2}(y)\setminus B_{\rho}(y)} \frac{\chi_E(x)-\chi_{E^c}(x)}{|x-y|^{n+s}}\,dx  \\
					&\qquad\qquad\geq \int_{B_{1/2}(y) \setminus B_{\rho}(y)} \frac{\chi_{E\setminus P}(x)}{|x-y|^{n+s}}\,dx
					+ \int_{B_{1/2}(y)} \frac{\chi_{P}(x)-\chi_{P^c}(x)}{|x-y|^{n+s}}\,dx =:\mathcal{I}_1+\mathcal{I}_2.
				\end{split}
			\end{equation}
			We stress that, since~$P$ is smooth, both $\mathcal{I}_1$ and $\mathcal{I}_2$ are well-defined. Now, we provide estimates for the right-hand side of~\eqref{eq::harnack_proof4}, starting from $\mathcal{I}_1$.

			Since~$y\in C_{\delta_0}\cap\left\{-a\leq x_n\leq -(1-\delta_0^2)a\right\}$, up to taking~$\delta_0$ (and hence~$a$) small enough, we have that~$C_{\delta_0,a}\subseteq B_{1/2}(y)$, and
			\begin{equation} \label{eq::harnack_proof5}
				|x-y|^{-n-s}\geq (3\delta_0)^{-n-s},
			\end{equation}
			for every~$x\in C_{\delta_0,a}$.
			
			By the density assumption in~\eqref{eq::half_measure_cylinder} and the coice of $\rho$, we also have that
			\begin{equation*}
				\begin{split}
					&|(E\setminus P)\cap B_\rho(y)\cap C_{\delta_0,a}| = |E\cap C_{\delta_0,a}|-|P\cap C_{\delta_0,a}| \\
					&\qquad\geq \frac{1}{2}|C_{\delta_0,a}|-|B_\rho(y)|-\left|B^{n-1}_{\delta_0}\times\{-a\leq x_n\leq -(1-\delta_0^2)a\}\right|\\
					&\qquad =\omega_{n-1}\delta_0^{n-1}a(1-\delta_0^2)-\omega_n\rho^n \geq \frac{1}{4}\omega_{n-1}\delta_0^{n-1}a.
				\end{split}
			\end{equation*}
			whenever~$\delta_0^2<1/2$.
			Consequently, using this together with~\eqref{eq::harnack_proof5}, we obtain that
			\begin{equation} \label{eq::harnack_proof6}
				\mathcal{I}_1 \geq \int_{C_{\delta_0,a}}\frac{\chi_{E\setminus P}(x)}{|x-y|^{n+s}}\,dx
				\geq \frac{|(E\setminus P)\cap C_{\delta_0,a}|}{(3\delta_0)^{n+s}}
				\geq c_2a\delta_0^{-1-s},
			\end{equation}
			for some~$c_2>0$ depending on~$n$ and~$s$.
			
			In addition, since~$P$ is smooth, we have (see for instance \cite{MR3107529, MR4279395})
			\begin{equation} \label{eq::harnack_proof7}
				\begin{split}
					&\mathcal{I}_2
					\geq c_3 \int_{0}^{1/2} r^{n-2} \left(\int_{0}^{-\frac{a}{2}r^2}\frac{dt}{(r^2+t^2)^{\frac{n+s}{2}}}\right) \,dr 
					\geq -c_3a,
				\end{split}
			\end{equation}
			for some~$c_3>0$ depending on~$n$ and~$s$, and possibly changing from line to line.
			
			Then, plugging~\eqref{eq::harnack_proof6} and~\eqref{eq::harnack_proof7} into~\eqref{eq::harnack_proof4}, we obtain
			\begin{equation}\label{eq::harnack_proof8}
				\int_{B_{1/2}(y)\setminus B_{\rho}(y)}  \frac{\chi_E(x)-\chi_{E^c}(x)}{|x-y|^{n+s}}\,dx \geq (c_2-c_3\delta_0^{1+s})a\delta_0^{-1-s} \geq c_4 a\delta_0^{-1-s},
			\end{equation}
			for some~$c_4>0$ depending on~$n$, $s$, and~$\alpha$.
			
			On the other hand, if~$E$ is a~$(2^{-k\gamma}\Lambda,\gamma)$-almost minimal set in~$B_{2^k}$, then~$2^{-k}E$ is a~$(\Lambda,\gamma)$-almost minimal set in~$B_{1}$. Moreover, there exists a parabula of opening $-2^{k-1}a=-2^{-1+k(1-\alpha)}$ touching $\partial(2^{-k}E)$ at~$2^{-k}y$. Let $r:=2^{-k(1-\alpha)}$ and $\zeta:=r^{1+\epsilon/2}=2^{-k}\rho$. Thus, by Theorem~\ref{cor::EL2}, we obtain that
			\begin{equation*}
				H_s^{\zeta}[2^{-k}E](2^{-k}y) \leq cr^{-s+\epsilon\frac{1-s}{2}} = c2^{k(1-\alpha)\left(s-\epsilon\frac{1-s}{2}\right)}, 
			\end{equation*}
			and hence
			\begin{equation} \label{eq::lower_bound_curv_rescaled}
				H_s^{\rho}[E](y) 
				= 2^{-ks}H_s^{\zeta}[2^{-k}E](2^{-k}y) \leq c2^{-\frac{k}{2}(\epsilon(1-s)-\epsilon\alpha(1-s)+2\alpha s)}.
			\end{equation}
			
			Therefore, from~\eqref{eq::harnack_proof3}, \eqref{eq::harnack_proof8}, and~\eqref{eq::lower_bound_curv_rescaled}, we infer that
			\begin{equation*} \label{eq::harnack_proof9}
				\big(c_4-c_1\delta_0^{1+s}\big)a\delta_0^{-1-s} = \big(c_4-c_1\delta_0^{1+s}\big)2^{-k\alpha}\delta_0^{-1-s} \leq c2^{-\frac{k}{2}(\epsilon(1-s)-\epsilon\alpha(1-s)+2\alpha s)}.
			\end{equation*}
			In particular, up to renaming the constants, we have
			\begin{equation*}
				\delta_0^{-1-s} \leq c 2^{-k\frac{1-s}{2}(\epsilon-\alpha\epsilon-2\alpha)},
			\end{equation*}
			with $\epsilon-\epsilon\alpha-2\alpha>0$. 
			
			However, this leads to a contradiction whenever we take~$k$ large enough and~$\delta_0$ small enough.
			\smallskip
			
			Furthermore, if~$\{x_n>a\}\cap C_1 \subseteq E$, up to considering the
			supergraph~$\widetilde{P}$ of the parabola~$x_n=y_n+\frac{a}{2}|x'-y'|^2$ tangent to~$E$ at some point~$y$ from above, the foregoing computations show that~\eqref{eq::inclusion_harnack2} holds true, concluding the proof.
		\end{proof}
		
		\begin{proof}[Proof of Theorem~\ref{th::harnack_ineq}]
			Suppose that~$|E\cap C_{\delta_0,a}|\geq \frac{1}{2}|C_{\delta_0,a}|$. Since~$E$ is~$(2^{-k\gamma}\Lambda,\gamma)$-minimal in~$B_{2^k}$, thanks to Lemma~\ref{lemma::partial_harnack}, we have either $C_{\delta_0}\cap\{x_n<-a(1-\delta_0^2)\}\subseteq E$ or~$C_{\delta_0}\cap\{x_n>a(1-\delta_0^2)\}\subseteq E$, from which it follows 
			
			\begin{align*}
				\mbox{either}\quad& \partial E\cap C_{\delta_0} \subseteq \big\{x_n\geq a(\delta_0^2-1)\big\}\\*
				\mbox{or}\quad& \partial E\cap C_{\delta_0} \subseteq \big\{x_n\leq a(1-\delta_0^2)\big\},
			\end{align*}
			respectively.
			
			On the other hand, if~$|E\cap C_{\delta_0,a}|< \frac{1}{2}|C_{\delta_0,a}|$, then necessarily~$|E^c\cap C_{\delta_0,a}|\geq \frac{1}{2}|C_{\delta_0,a}|$.
			
			Moreover, since $E$ is an almost minimal, we have that also~$E^c$ is almost minimal. Thus, using Lemma~\ref{lemma::partial_harnack} with~$E^c$, we deduce again that
			
			\begin{align*}
				\mbox{either}\quad& \partial E\cap C_{\delta_0} \subseteq \big\{x_n\geq a(\delta_0^2-1)\big\}\\
				\mbox{or}\quad& \partial E\cap C_{\delta_0} \subseteq \big\{x_n\leq a(1-\delta_0^2)\big\},
			\end{align*}
			as desired.
		\end{proof}
		
		An  iterative application of Theorem~\ref{th::harnack_ineq} leads to the following:
		\begin{lemma}[{H\"older estimates, \cite[Lemma 2.20]{MR5015001}}] \label{lemma::holder_est}
			Let~$\{E_k\}_k$ be a sequence of~$(2^{-k\gamma}\Lambda,\gamma)$-almost minimal sets. 
			
			Suppose that, for every $k\in\N$, there exists~$h_k\in\N$ such that
			\begin{equation} \label{eq::flatness_Ek}
				\partial E_k\cap C_{2^h} \subseteq \big\{|x\cdot\nu_h^k|\leq a_k2^{h(1+\alpha)}\big\},\quad{\mbox{for all }}h\in\{ 0,\dots, h_k\} ,
			\end{equation} 
			with~$\nu_0^k=e_n$ for every~$k$. 
			
			For any~$r\geq1$, let
			$$ A_k^r:=\big\{(x',x_n/a_k)\;{\mbox{ s.t. }}\; (x',x_n)\in\partial E_k\cap C_r\big\} \quad\text{ for every } k\in\N.$$
			Assume that there exists a positive constant~$C$ independent of~$k$ such that~$A_k^r\subseteq \{|x_n|\leq C\}$, for all~$k\in\N$.
			
			Then, there exists~$k_0\in\N$, depending only on~$n$,~$s$,~$\Lambda$,~$\gamma$, and~$\alpha$, such that, for all~$ k\geq k_0$, 
			\begin{equation}\label{eq::holder_est}
				A_k^r\cap\{|x'|<r/2\}\subseteq \subgr(g_k^+)\setminus \subgr(g_k^-),
			\end{equation}
			where~$g_k^-$ and~$g_k^+$ are H\"older continuous functions with modulus of continuity~$\theta t^\eta$, for some~$\eta\in(0,1)$ and~$\theta>0$ depending only on~$C$ above and~$\delta_0$ from Theorem~\ref{th::harnack_ineq}.
		\end{lemma}
		
		For the convenience of the reader, a proof of Lemma~\ref{lemma::holder_est} can be found in Appendix~\ref{sec::holder_est_proof}. We point out that the argument does not dependent explicitly on $\gamma$.
		
		\begin{rem} \label{rem::unif_conv_holder}
			By inspection of the proof of Lemma~\ref{lemma::holder_est}, given any $x\in A_k^r\cap \left\{|x'|<\frac{r}{2}\right\}$, we have
			$$ 0<g_k^+(x')-g_k^-(x')\leq 2\theta b_k^\eta ,$$
			for some H\"older modulus of continuity $b_k$. Therefore, since~$b_k\to0$, as~$k\to+\infty$, we deduce from Ascoli-Arzel\`a's Theorem that, $g_k^-$ and~$g_k^+$ uniformly converge to some~$\cont^{0,\eta}$-function~$u$ (up to subsequences). Hence, $A_k^r$ converges uniformly to~$\mbox{graph}(u)$ in~$\{|x'|<r/2\}$.
		\end{rem}
		
		\subsection{Linearized equation for the H\"older limit in Lemma~\ref{lemma::holder_est}} \label{sec::linearized equation}
		
		As last ingredient needed for the proof of Proposition~\ref{prop::improv_flat_rescaled}, we establish the following partial results. We will show that the limit function $u$ from Lemma~\ref{lemma::holder_est} is $\frac{1+s}{2}$-harmonic, and consequently linear, contradicting the claim in~\eqref{eq::claim_contradiction}. The heuristics behind this argument is that the sequence of almost minimizers $\{E_k\}_k$ defined in \eqref{eq::flatness_Ek2} still satisfies a limiting equation (as $k\to+\infty$) whose linearization coincides with $\frac{1+s}{2}$-fractional Laplace equation. The statements go as follows:
		
		\begin{lemma}[Lemma 6.10, \cite{MR2675483}]\label{lemma::conv_unif}
			There exists a subsequence~$\{E_{k_j}^*\}_j$ that is uniformly convergent on every compact subset of~$\R^n$ to the subgraph of a H\"older continuous function~$u$ such that~$u(0')=0$, and for which there exists some~$c>0$ such that~$u(x')\leq c(1+|x'|^{1+\alpha})$, for every~$x'\in\R^{n-1}$.
		\end{lemma}
		
		\begin{lemma}[Lemma 6.11, \cite{MR2675483}]\label{lemma::linear_viscosity}
			Let~$u$ be a function such that the sequence of almost-minimizers~$\{E_k\}_k$ converges to~$\subgr(u)$ uniformly on every compact subset of~$\R^n$, as~$k\to+\infty$.
			
			Then, 
			$$ (-\Delta)^{\frac{1+s}{2}} u=0 \text{ in }\R^{n-1},$$
			in the viscosity sense.
		\end{lemma}
		
		\begin{lemma}[Proposition 6.7, \cite{MR2675483}] \label{lemma::CRS_harmonic_implies_linear}
			If~$(-\Delta)^{\frac{1+s}{2}} u=0$ in~$\R^{n-1}$ in the viscosity sense, and 
			$$|u(x)|\leq c\big(1+|x|^{1+\alpha}\big),$$ 
			then~$u$ is linear.
		\end{lemma}
		
		We refer to \cite{MR2675483} and \cite{MR5015001} for the proofs of Lemmata~\ref{lemma::conv_unif} and \ref{lemma::CRS_harmonic_implies_linear}, respectively. On the other hand, since few modifications are needed in the proof of Lemma~\ref{lemma::linear_viscosity}, we present the details here below.
		
		
		We prove that the function~$u$ from Lemma~\ref{lemma::holder_est} is a viscosity solution of 
		\begin{equation} \label{eq::viscosity_harmonic}
			(-\Delta)^{\frac{1+s}{2}} u=0
		\end{equation}
		in the sense of~\cite[Definition~2.2]{MR2494809}. More precisely, we show that~$u$ is both viscosity super- and subsolution of~\eqref{eq::viscosity_harmonic}.
		
		\begin{proof}
			We start showing that $u$ is a viscosity supersolution of~\eqref{eq::viscosity_harmonic}. To this purpose, let~$\psi$ be a function touching~$u$ from below at some point~$x_0$. Without loss of generality, we assume that~$x_0=0$. Moreover assume that~$\psi\in\cont^2(\overline{B}_\xi)$, for some~$\xi>0$, and~$\psi<u$ in~$B_\xi\setminus\{0\}$. 
			
			We point out that, in order to prove that~$u$ is a viscosity supersolution of~\eqref{eq::viscosity_harmonic}, it is sufficient to show that
			\begin{equation}\label{eq::lin_visc_claim}
				\int_{\R^{n-1}} \frac{u(x')-u(0')}{|x'|^{n+s}}\,d\haus{n-1}(x) \leq 0,
			\end{equation}
			since 
			\begin{equation*}
				\begin{split}
					&	 \int_{\R^{n-1}}  \frac{u(x')-u(0')}{|x'|^{n+s}}\,d\haus{n-1}(x) 
					\\&\qquad\geq \int_{B_\xi} \frac{\psi(x')-\psi(0')}{|x'|^{n+s}}\,d\haus{n-1}(x)  
					+ \int_{\R^{n-1}\setminus B_\xi} \frac{u(x')-u(0')}{|x'|^{n+s}}\,d\haus{n-1}(x) .
				\end{split}
			\end{equation*}
			Thus, we focus on proving~\eqref{eq::lin_visc_claim}.
			
			Thanks to Lemma~\ref{lemma::conv_unif}, for any given~$\eta>0$ and~$R>0$, by definition of uniform convergence, there exists~$k_\eta\in\N$ such that, for every~$k\geq k_\eta$, $\partial E_k\cap B_R$ lies in a~$a_k\eta$-neighborhood of~$\mbox{graph}(a_ku)$, where~$a_k:=2^{-\alpha k}$. 
			
			Furthermore, since $u(0')=0$ and $|\nabla u(0')|=0$, there exists a vertical translation of a parabola of opening~$-\frac{a_k}{2}$ and tangent to~$\partial E_k$ in~$B_\eta$ at some point~$x_k$.
			Besides, $x_k\to0$ and $k_\eta\to+\infty$ as~$\eta\to0$.
			
			Now, for every~$\overline{x}=(\overline{x}',\overline{x}_n)\in\R^{n-1}\times\R$ and for every~$r>0$, we define the truncated closed cylinder
			$$ D_r(\overline{x}):=\big\{x\in\R^n\;{\mbox{ s.t. }}\;|x'-\overline{x}'|\leq r,\; |x_n-\overline{x}_n|\leq r\}.$$
			Let also also consider a small $\delta>0$ and a large $R>0$.
			
			By the flatness assumptions in~\eqref{eq::flatness_Ek}, we know that
			\begin{equation*}
				\partial E_k\cap D_R(x_k)\subseteq\partial E_k\cap B_{2R}(x_k)\subseteq\big\{|x_n-x_{k,n}|\leq c_R a_k\big\},
			\end{equation*} 
			where~$c_R$ is a positive constant depending on~$R$.
			Up to taking~$k$ large enough, we assume that~$c_Ra_k<\delta/2$. Then, for every~$x\in \big(D_R(x_k)\setminus D_\delta(x_k)\big) \cap \{|x_n-x_{k,n}|\leq c_R a_k\}$,
			\begin{equation*}
				|x'-x_k'| \geq \delta^2 - |x_n-x_{k,n}|^2 \geq \delta^2 - c_R^2a_k^2 \geq \frac{3}{4}\delta^2.
			\end{equation*}
			
			Moreover, we observe that
			\begin{equation*}
				\begin{split}
					\left|\frac{1}{|x-x_k|^{n+s}}-\frac{1}{|x'-x_k'|^{n+s}}\right| 
					&= \frac{n+s}{2} \int_{0}^{1} \frac{|x_n-x_{k,n}|^2}{(|x'-x_k'|^2+t|x_n-x_{k,n}|^2)^{\frac{n+s+2}{2}}} \,dt \\
					&\leq \frac{n+s}{2} c_R^2a_k^2 \int_{0}^{1} \frac{dt}{|x'-x_k'|^{n+s+2}} .
				\end{split}
			\end{equation*}
			It follows from the last two displays that
			\begin{equation*}
				\left|\frac{1}{|x-x_k|^{n+s}}-\frac{1}{|x'-x_k'|^{n+s}}\right|  \leq Ca_k^2 ,
			\end{equation*}
			for some positive constant~$C$, depending on~$n$, $s$, $R$, and~$\delta$.
			
			Using this and the fact that~$\partial E_k$ lies in a~$a_k\eta$-neighborhood of~$a_k\mbox{graph}(u)$, we arrive at 
			\begin{equation}\label{eq::lin_visc1}
				\begin{split}
					&\int_{D_R(x_k)\setminus D_\delta(x_k)} \frac{\chi_{E_k}(x)-\chi_{E_k^c}(x)}{|x-x_k|^{n+s}}\,dx \\
					&\qquad= 2a_k\int_{B_R^{n-1}(x_k)\setminus B_\delta^{n-1}(x_k)} \frac{u(x')-u(x_k')+O(\epsilon)}{|x'-x_k'|^{n+s}}\,d\haus{n-1}(x) + O(a_k^3)\\
					&\qquad= 2a_k\int_{B_R^{n-1}(x_k)\setminus B_\delta^{n-1}(x_k)} \frac{u(x')-u(x_k')}{|x'-x_k'|^{n+s}} \,d\haus{n-1}(x) + O(\eta a_k) +O(a_k^3).
				\end{split}
			\end{equation}
			We now write the set~$D_R(x_k)\setminus D_\delta(x_k)$ as
			\begin{equation*}
				B_\delta^c(x_k)\setminus\left[D_R(x_k)^c\cup \left(D_\delta(x_k)\cap B_\delta^c(x_k)\right)\right].
			\end{equation*}
			Thus, we have
			\begin{equation} \label{eq::lin_visc2}
				\begin{split}
					&\int_{D_R(x_k)\setminus D_\delta(x_k)} \frac{\chi_{E_k}(x)-\chi_{E_k^c}(x)}{|x-x_k|^{n+s}}\,dx 
					= \int_{B_\rho^c(x_k)} \frac{\chi_{E_k}(x)-\chi_{E_k^c}(x)}{|x-x_k|^{n+s}}\,dx \\
					&\qquad\qquad - \int_{D_R^c(x_k)} \frac{\chi_{E_k}(x)-\chi_{E_k^c}(x)}{|x-x_k|^{n+s}}\,dx
					- \int_{D_\delta(x_k)\setminus B_\rho(x_k)} \frac{\chi_{E_k}(x)-\chi_{E_k^c}(x)}{|x-x_k|^{n+s}}\,dx ,
				\end{split}
			\end{equation}
			and we prove suitable estimates for the right-hand side.
			
			To this end, recalling the estimates in~\eqref{eq::harnack_proof3}, we have
			\begin{equation}\label{eq::lin_visc3}
				\begin{split}
					\left|\int_{D_R^c(x_k)} \frac{\chi_{E_k}(x)-\chi_{E_k^c}(x)}{|x-x_k|^{n+s}}\,dx\right|
					&\leq C'a_k \int_{R}^{+\infty} \rho^{\alpha-1-s}\,d\rho + \frac{\omega_n}{s}2^{-ks}\\
					&\leq C'\big(a_kR^{\alpha-s}+ a_k^{1+\mu}\big),
				\end{split}
			\end{equation}
			for some positive exponent~$\mu$, and some positive constant~$C'$,
			depending only on~$n$, $s$, and~$\alpha$, and possibly changing from line to line.
			
			To estimate the term involving $D_\delta(x_k)\setminus B_\rho(x_k)$, let~$P_k$ be the subgraph of the tangent parabola~$x_{k,n}-\frac{a_k}{2}|x'-x_k'|^2$. Since $P_k$ is smooth, the non-local mean curvature of~$P_k$ is well-defined at~$x_k$, for every~$k$. Let also $\rho:=2^{-\frac{k}{2}(\epsilon-\epsilon\alpha-2\alpha)}$. Up to taking $k$ large enough, we have that $\rho<\delta$.
			Then, arguing as in~\eqref{eq::harnack_proof7}, we obtain
			\begin{equation} \label{eq::lin_visc4}
				\begin{split}
					&\int_{D_\delta(x_k)\setminus B_\rho(x_k)} \frac{\chi_{E_k}(x)-\chi_{E_k^c}(x)}{|x-x_k|^{n+s}}\,dx 
					\geq \int_{D_\delta(x_k)\setminus B_\rho(x_k)} \frac{\chi_{P_k}(x)-\chi_{P_k^c}(x)}{|x-x_k|^{n+s}}\,dx \\
					&\qquad\qquad\geq -\int_{0}^{\delta}  \frac{Ca_kr^n}{r^{n+s}} dr 
					= -Ca_k\delta^{1-s} ,
				\end{split}
			\end{equation}
			for a positive constant~$C$ independent of~$k$.
			
			Therefore, using~\eqref{eq::lin_visc2},~\eqref{eq::lin_visc3} and~\eqref{eq::lin_visc4} in~\eqref{eq::lin_visc1}, we obtain that
			\begin{equation} \label{eq::lin_visc_xk_estimate}
				\begin{split}
					&\int_{B_R^{n-1}(x_k)\setminus B_\delta^{n-1}(x_k)} \frac{u(x')-u(x_k')}{|x'-x_k'|^{n+s}}\,d\haus{n-1}(x) \\
					& = \frac{1}{2a_k} \int_{D_R(x_k)\setminus D_\delta(x_k)} \frac{\chi_{E_k}(x)-\chi_{E_k^c}(x)}{|x-x_k|^{n+s}}\,dx +O(\eta) +O(a_k^2)\\
					& \leq 
					\frac{1}{2a_k}\int_{B_\rho^c(x_k)} \frac{\chi_{E_k}(x)-\chi_{E_k^c}(x)}{|x-x_k|^{n+s}}\,dx 
					+C\delta^{1-s}+ O(\eta) +O(a_k^{\mu})+O(R^{\alpha-s}).
				\end{split}
			\end{equation}
			
			Thanks to the assumptions on~$E_k$ and Theorem~\ref{th::EL} (see also~\eqref{eq::lower_bound_curv_rescaled}), we know that 
			\begin{equation*}
				\frac{1}{2a_k}\int_{B_\rho^c(x_k)} \frac{\chi_{E_k}(x)-\chi_{E_k^c}(x)}{|x-x_k|^{n+s}}\,dx \leq ca_k^{-1}2^{-\frac{k}{2}(\epsilon(1-s)-\epsilon\alpha(1-s)-2\alpha s)} = c2^{-k\frac{1-s}{2}(\epsilon-\epsilon\alpha-2\alpha)} .
			\end{equation*}
			
			It thereby follows that
			\begin{equation}\label{eq::lin_visc_xk_estimate_2}
				\begin{split}
					&\int_{B_R^{n-1}(x_k)\setminus B_\delta^{n-1}(x_k)} \frac{u(x')-u(x_k')}{|x'-x_k'|^{n+s}}\,d\haus{n-1}(x) \\
					&\qquad\qquad\leq c2^{-k\frac{1-s}{2}(\epsilon-\epsilon\alpha-2\alpha)} + C\delta^{1-s}+ O(\eta) +O(a_k^{\mu})+O(R^{\alpha-s}).
				\end{split}
			\end{equation}
			Since~$\alpha\in\left(0,\frac{\epsilon}{2+\epsilon}\right)$, taking the limit as~$\eta\to0$, and hence as~$k\to+\infty$, in~\eqref{eq::lin_visc_xk_estimate_2}, we infer that
			\begin{equation*}
				\int_{B_R^{n-1}\setminus B_\delta^{n-1}} \frac{u(x')-u(0')}{|x'|^{n+s}}\,d\haus{n-1}(x) \leq C\delta^{1-s}+O(R^{\alpha-s}).
			\end{equation*}
			Hence, taking the limits as~$\delta\to0$ and~$R\to+\infty$, we obtain~\eqref{eq::lin_visc_claim}.
			\smallskip
			
			A similar argument to the one above shows that~$u$ is also a viscosity subsolution of~\eqref{eq::viscosity_harmonic}. Let~$\phi$ be a~$\cont^2$-function touching~$u$ at~$0$ from above. Using the same idea in~\eqref{eq::lin_visc2} with a tangent parabola of opening~$\frac{a_k}{2}$, we deduce that
			\begin{equation*}
				\int_{D_\delta(x_k)\setminus B_\rho(x_k)} \frac{\chi_{E_k}(x)-\chi_{E_k^c}(x)}{|x-x_k|^{n+s}}\,dx \leq Ca_k\delta^{1-s} .
			\end{equation*}
			This, together with~\eqref{eq::lin_visc1}, \eqref{eq::lin_visc3}, and Corollary~\ref{cor::reverseELeq}, entails
			\begin{equation} \label{eq::lin_visc5}
				\begin{split}
					&\int_{B_R^{n-1}(x_k)\setminus B_\delta^{n-1}(x_k)} \frac{u(x')-u(x_k')}{|x'-x_k'|^{n+s}}\,d\haus{n-1}(x) \\
					& \geq 
					\frac{1}{2a_k}\int_{B_\rho^c(x_k)} \frac{\chi_{E_k}(x)-\chi_{E_k^c}(x)}{|x-x_k|^{n+s}}\,dx 
					-C\delta^{1-s}+ O(\eta) +O(a_k^{\mu})+O(R^{\alpha-s})\\
					&\geq -c2^{-k\frac{1-s}{2}(\epsilon-\epsilon\alpha-2\alpha)} - C\delta^{1-s}+ O(\eta) +O(a_k^{\mu})+O(R^{\alpha-s}).
				\end{split}
			\end{equation}
			Taking~$\eta\to0$, and consequently~$k\to+\infty$, in~\eqref{eq::lin_visc5}, we obtain that
			\begin{equation*}
				\int_{B_R^{n-1}\setminus B_\delta^{n-1}} \frac{u(x')-u(0')}{|x'|^{n+s}}\,d\haus{n-1}(x) \geq -C\delta^{1-s}+O(R^{\alpha-s}).
			\end{equation*}
			Finally, taking the limits as~$\delta\to0$ and~$R\to+\infty$, we infer that~$u$ is also a viscosity subsolution of~\eqref{eq::viscosity_harmonic}.
		\end{proof}
		
		\subsection{Proof of the flatness improvement (Proposition~\ref{prop::improv_flat_rescaled})}
		Arguing by contradiction, suppose that the conclusion in~\eqref{eq::flatness_1/2} does not hold true. Then, we construct a sequence~$\{E_k\}_k$ such that 
		\begin{equation}
			\label{iu6y54trstar463782}
			\begin{split}
				&\partial E_k\cap B_{2^j}\subseteq \big\{|x\cdot\nu_j|\leq2^{j+\alpha(j-k)}\big\},\quad {\mbox{for all }}j\in\{0,\dots, j_k\},\\
				\text{but}\quad &\partial E_k\cap B_{1/2}\not\subseteq \big\{|x\cdot\nu_{-1}|\leq2^{-1-\alpha-\alpha k}\big\}.
			\end{split}
		\end{equation}
		Thanks to Lemma~\ref{lemma::conv_unif}, we have that $\{E_k\}_k$ converges uniformly in~$B_1$ to the subgraph of some H\"older-continuous function~$u$ (up to subsequences). By Lemma~\ref{lemma::linear_viscosity}, $u$ is linear and $u(0')=0$, hence there exists $ \nu\in \S^{n-1}$ such that
		$$ E:= \subgr(u) = \{x\cdot\nu\leq0\}. $$
		Since the sequence~$\{E_k\}_k$ is uniformly convergent, $E_k$ lies in an~$\eta$-neighborhood of~$E$ in~$B_{1/2}$, for every~$k$ sufficiently large. Thus, up to making~$\eta$ sufficiently small, there exists some~$\overline{k}$ such that
		$$ \partial E_{\overline{k}}\cap B_{1/2}\subseteq \big\{|x\cdot\nu |\leq2^{-1-\alpha-\alpha \overline{k}}\big\} . $$
		However, this contradicts~\eqref{iu6y54trstar463782}, concluding the proof.

		
		\section{H\"older regularity for almost minimal sets} \label{sec::holder_almost_min}
		
		\subsection{Proof of Theorem~\ref{th::holder_reg_almost_min}} \label{sec::proof_holder_reg}
		Now, we show that every almost minimizer in~$B_1$ is a $\cont^{1,\alpha}$-surface in~$B_{1/2}$, for every~$\alpha\in(0,\epsilon_0)$, where $\epsilon_0:=\frac{\epsilon}{2+\epsilon}$ and $\epsilon$ is as in Theorem~\ref{th::EL}.
		
		To this purpose, consider~$x_0\in\partial E\cap B_{1/2}$. Let~$k_0$ be as in Theorem~\ref{th::flatness_improvement} and choose~$\eta_0<2^{-k_0(1+\alpha)}$. In this way, the improvement of flatness result (Theorem~\ref{th::flatness_improvement}) holds for the translated set~$E-x_0$. In particular, there exists a sequence~$\{\nu_j\}_j$ of unit vectors such that
		\begin{equation} \label{eq::flat_inclusions}
			\partial E\cap B_{2^{-j}}(x_0) \subseteq \big\{|(x-x_0)\cdot\nu_j|\leq 2^{-j(1+\alpha)}\big\},
		\end{equation}
		for every~$j\in\{0,\dots,k_0\}$, with~$\nu_0=e_n$ . By~\eqref{eq::flat_inclusions} and a standard geometric argument (see for instance Section G in the appendix of~\cite{MR5015001}), we obtain the estimate
		\begin{equation} \label{eq::unit_normal_estimate}
			|\nu_j-\nu_{j+1}|\leq c2^{-j\alpha}.
		\end{equation}
		Thus, the sequence~$\{\nu_j\}_j$ converges to some unit vector~$\nu(x_0)$ and
		\begin{equation*}
			\begin{split}
				&|\nu(x_0)-\nu_j|
				=\lim_{k\to+\infty}|\nu_{j+k}-\nu_{j}| \leq \lim_{k\to+\infty} \sum_{q=0}^{k-1}|\nu_{j+q+1}-\nu_{j+q}|\\
				&\qquad\qquad\leq \lim_{k\to+\infty} \sum_{q=0}^{k-1} 2^{-(j+q)\alpha} \leq c2^{-j\alpha}.
			\end{split}
		\end{equation*}
		In view of this, we obtain that
		$$ |(x-x_0)\cdot\nu(x_0)| \leq |(x-x_0)\cdot\nu_j|+|x-x_0|\,|\nu(x_0)-\nu_j|\leq c2^{-j(1+\alpha)} ,$$
		for every~$x\in\partial E\cap B_{2^{-j}}(x_0)$, namely
		\begin{equation}\label{eq::flat_nu0}
			\partial E\cap B_{2^{-j}}(x_0) \subseteq \big\{|(x-x_0)\cdot\nu(x_0)|<c2^{-j(1+\alpha)}\big\}.
		\end{equation}
		
		Now, for every points~$x_0$, $y_0\in\partial E\cap B_{1/2}$ such that~$x_0\in B_{2^{-j}}(y_0)\setminus B_{2^{-j-1}}(y_0)$, for some~$j\in\N$, we estimate the angle between the unit normals~$\nu(x_0)$ and~$\nu(y_0)$. Thanks to the flatness condition~\eqref{eq::flat_nu0}, we have that
		\begin{align*}
			\text{either }&\big\{(x-y_0)\cdot\nu(x_0)>-c2^{-j(1+\alpha)}\big\}\subseteq \big\{x\cdot\nu(y_0)<c2^{-j(1+\alpha)}\big\}\\
			\text{or }&\big\{x\cdot\nu(y_0)>-c2^{-j(1+\alpha)}\big\}\subseteq\big\{(x-y_0)\cdot\nu(x_0)<c2^{-j(1+\alpha)}\big\} \text{ in }B_{2^{-j}}(y_0).
		\end{align*}
		Therefore, arguing as for~\eqref{eq::unit_normal_estimate}, we infer the worst possible estimate
		$$ |\nu(x_0)-\nu(y_0)|\leq C2^{-j\alpha}\leq 2C|x_0-y_0|^\alpha.$$
		Hence, ~$\nu$ is H\"older continuous, which yields that~$\partial E\cap B_{1/2}$ is a~$\cont^{1,\alpha}$ surface (see for instance~\cite[Theorem~5.8]{MR0165073}).
		
		\subsection{Proof of Corollary~\ref{cor::optimal_reg}} \label{sec::better_regularity}
		
		Now, we show that with Theorem~\ref{th::holder_reg_almost_min} in hand, we are able to improve the estimate in Theorem~\ref{th::EL}, and consequently we infer a better H\"older regularity for almost minimal sets.  
		The main tool of our strategy is the following technical lemma, which provide sharper uniform density estimates for $\cont^{1,\alpha}$ sets.
		
		\begin{lemma} \label{lemma::holder_ude}
			Let $E$ be a $\cont^{1,\alpha}$-regular set in $\R^n$, for some $\alpha\in(0,\epsilon_0)$, where $\epsilon_0$ is as in Theorem~\ref{th::holder_reg_almost_min}. Suppose that the ball $B_{2r}(-2r e_n)$ touches $\partial E$ from below at the origin, for some small $r>0$, and let $A_d$ be defined as in~\eqref{eq::def_bubble} with $d=r^{1+2\epsilon}$ for some $\epsilon>0$.
			
			Then, there exists $c_0$ depending only on $n$,~$s$,~$\Lambda$, and~$\gamma$ such that
			\begin{equation} \label{eq::holder_ude}
				|A_d| \geq  c_0 r^{n(1+2\epsilon)-\frac{\alpha}{1+\alpha}(1+\epsilon)(n-1)} 
			\end{equation}
		\end{lemma}
		
		Arguing as in Section~\ref{sec::improvement_of_flatness}, and employing Lemma~\ref{lemma::holder_ude} in place of Proposition~\ref{prop::UDE} in the proof of Theorem~\ref{th::EL}, we infer the following:
		
		\begin{cor} \label{cor::better_reg}
			Let $\Lambda\geq0$ and $\gamma$. Let $E$ be a $\cont^{1,\alpha}$-regular $(\Lambda,\gamma)$-almost minimal set in $B_1$, for every $\alpha\in(0,\epsilon_0)$, for some $\epsilon_0\in(0,s)$.
			
			Then, $E$ is $\cont^{1,\alpha}$-regular, for any $\alpha\in(0,\epsilon_1)$, with
			\begin{equation*}
				\epsilon_1 = \frac{\gamma+\epsilon_0(n+\gamma-1)}{2n+1+s-\gamma +\epsilon_0(n+s-\gamma)}.
			\end{equation*}
		\end{cor}
		\begin{rem}
			We emphasize that Lemma~\ref{lemma::holder_ude} allows a better choice of $\epsilon$ in~\eqref{eq::epsilon_choice}, which in turn improves~\eqref{cor::EL2} in Theorem~\ref{th::EL}. Recall that $\epsilon_0:=\frac{\epsilon}{2+\epsilon}$, where $\epsilon$ is precisely given by~\eqref{eq::epsilon_choice}. Then, we obtain $\epsilon_0<\epsilon_1$.
		\end{rem}
		
		\begin{proof}[Proof of Lemma~\ref{lemma::holder_ude}]
			Without loss of generality,we can suppose that $E$ coincides with the subgraph
			\begin{equation*}
				\{(x',x_n)\in\R^n \mbox{ s.t. }x_n<|x'|^{1+\alpha}\}
			\end{equation*}
			around the origin. Then, we estimate the volume of the set $A_d$ defined as in~\eqref{eq::def_bubble} in the worst case scenario (which is precisely when the concavity of $\partial E$ is facing upwards).
			\begin{figure}[h!]
				\centering
				\includegraphics[width=0.4\linewidth]{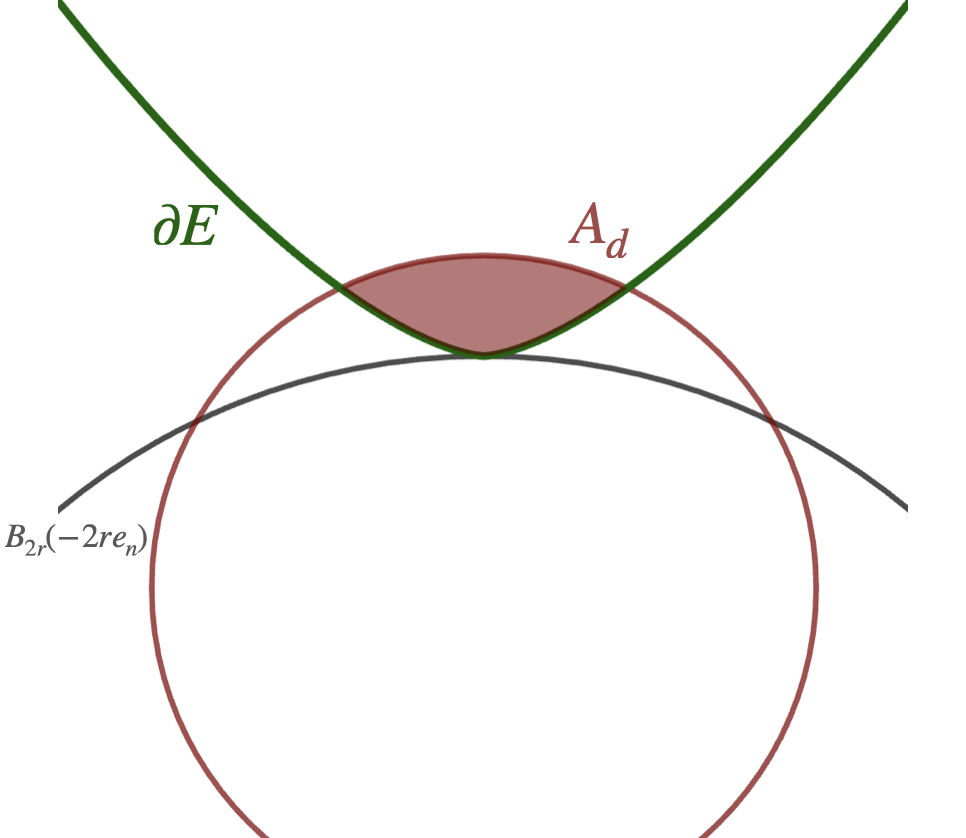}
				\caption{The set $A_d$ when the $E=\{x_n<|x'|^{1+\alpha}\}$.}
			\end{figure}
			
			To find the intersection between $\partial E$ and $\partial A_d$, we solve the equation
			\begin{equation} \label{eq::intersection}
				d-\frac{-|x'|^2}{2r} = |x'|^{1+\alpha}.
			\end{equation}
			
			Let us assume that a solution of~\eqref{eq::intersection} has the form $|x'|=r^\beta$, with $\beta\geq0$. Then, recalling that $d=r^{1+2\epsilon}$, we find that $\beta$ satisfies 
			\begin{equation} \label{eq::intersection2}
				r^{2\beta}+r^{2\beta(1+\alpha)} = 2r^{2+2\epsilon}.
			\end{equation}
			
			Since $r<1$ is small, comparing the terms on both sides of~\eqref{eq::intersection2}, we deduce that the dominant term on the right hand side must be of order $2+2\epsilon$.
			Let us suppose that $r^{2\beta}\geq r^{1+\beta(1+\alpha)}$, and hence $\beta=1+\epsilon$. Recalling that $\alpha<\frac{\epsilon}{2+\epsilon}$, we infer that
			\begin{equation*}
				1+\beta(1+\alpha) = 2+\epsilon+\alpha+\alpha\epsilon < 2+2\epsilon,
			\end{equation*}
			against $r^{2\beta}\geq r^{1+\beta(1+\alpha)}$. It thereby follows that $r^{2\beta}\leq r^{1+\beta(1+\alpha)}$ and
			\begin{equation*}
				\beta = \frac{1+2\epsilon}{1+\alpha}.
			\end{equation*}
			
			Therefore, we conclude that
			\begin{equation*}
				\begin{split}
					&|A_d| \geq \int_{B_{r^\beta}^{n-1}} \left(d-\frac{-|x'|^2}{2r} - |x'|^{1+\alpha}\right) dx' \\
					&\qquad\qquad= \omega_n\left(dr^{\beta} -\frac{1}{2(n-1)}r^{(n+1)\beta-1}-\frac{1}{n+\alpha}r^{(n+\alpha)\beta}\right) \\
					&\qquad\qquad \geq \frac{n+\alpha-2}{2(n+\alpha)} r^{(1+2\epsilon)\frac{n+\alpha}{1+\alpha}}=: c_0 r^{n(1+2\epsilon)-\frac{\alpha}{1+\alpha}(1+\epsilon)(n-1)}. 
				\end{split}
			\end{equation*}
			as desired.
		\end{proof}
		
		\begin{proof}[Proof of Corollary~\ref{cor::optimal_reg}]
			Thanks to Theorem~\ref{th::holder_reg_almost_min}, we know that $E$ is $\cont^{1,\alpha}$-regular for every $\alpha\leq \epsilon_0$. From here, by virtue of Lemma~\ref{lemma::holder_ude} and Corollary~\ref{cor::better_reg}, we improve the regularity of $E$ up to 
			\begin{equation*}
				\epsilon_1 = \frac{\gamma+\epsilon_0(n+\gamma-1)}{2n+1+s-\gamma +\epsilon_0(n+s-\gamma)}.
			\end{equation*}
			
			Iterating this procedure, we improve the regularity for $E$, and we find a sequence $\{\epsilon_j\}_j$ such that
			\begin{equation*}
				\epsilon_{j+1} =  \frac{\gamma+\epsilon_j(n+\gamma-1)}{2n+1+s-\gamma +\epsilon_j(n+s-\gamma)}.
			\end{equation*}
			
			Since the sequence $\{\epsilon_j\}_j$ s increasing with respect ot the parameter $j$, there exists 
			\begin{equation*}
				\overline{\epsilon} :=\lim_{j\to+\infty} \epsilon_j,
			\end{equation*}
			and it satisfies
			\begin{equation*}
				\overline{\epsilon} = \frac{\gamma+\overline{\epsilon}(n+\gamma-1)}{2n+1+s-\gamma +\overline{\epsilon}(n+s-\gamma)},
			\end{equation*}
			
			In particular, we have
			\begin{equation*} \label{eq::optimal_exp}
				\overline{\epsilon} = \frac{-(n+s+2-2\gamma)+\left[ (n+s+2-2\gamma)^2+4\gamma(n+s-\gamma)\right]^{1/2}}{2(n+s-\gamma)}.\qedhere
			\end{equation*}
			
		\end{proof}
		
		
		\section{Hausdorff dimension of the singular set (Theorem~\ref{th::haus_dim_singular_almost})} \label{sec::proof_haus_dim_sing}
		
		By mean of to the monotonicity formula from Proposition~\ref{prop::Phi_monotone}, we are able to study regular and singular points(in the sense of Definition~\ref{def::sing_reg_points}) of~$E$ through a blow-up analysis. In particular, if~$E$ is an almost minimal set in~$B_1$ such that~$0\in\partial E$, we consider the dilation~$\lambda E$ and study the limit set when~$\lambda\to+\infty$. We will show that such a limit exists and is a $s$-minimal cone. in the following sense:
		\begin{definition}
			We say that a set~$C$ is a $s$-minimal cone if it is a~$0$-homogeneous local minimizer for the~$s$-perimeter, i.e.~$C$ is a minimizer for~$\Per_s(\cdot,B_1)$ such that~$\lambda C=C$, for every~$\lambda>0$.
		\end{definition}
		
		The main tool in our proof is the following:
		
		\begin{lemma}[{\cite[Proposition~9.1]{MR2675483}}] \label{lemma::phi_continuous_in_E}
			Let~$\{\lambda_k\}_k$ and $\{\Lambda_k\}_k$ be sequences of positive real numbers such that $\lambda_k\to+\infty$ and $\Lambda_k\to\Lambda\in[0,+\infty)$, as~$k\to+\infty$, and let $\gamma\in(0,s]$. 
			
			Let~$\{E_k\}_k$ be a sequence of~$(\Lambda_k,\gamma)$-minimal sets in~$B_{\lambda_k}$. Moreover, assume that~$E_k\to E$ in~$L^1_{\loc}$, for some set~$E$
			of finite fractional perimeter, as~$k\to+\infty$.
			
			Let also~$u_k:=\chi_{E_k}-\chi_{E_k^c}$ and~$u:=\chi_E-\chi_{E^c}$, with extension~$\widetilde{u}_k$ and~$\widetilde{u}$, respectively.
			
			Then,
			\renewcommand{\theenumi}{\roman{enumi}}
			\begin{enumerate}
				\item\label{item::phi_continuous_in_E_1} $\widetilde{u}_k\to\widetilde{u}$ uniformly on every compact set~$K\Subset\R^{n+1}_+$;
				\item\label{item::phi_continuous_in_E_2} $\nabla\widetilde{u}_k\to\nabla\widetilde{u}$ in~$L^2_{\loc}(\R^{n+1}_+,z^{1-s}\,dx\,dz)$.
			\end{enumerate}
			In particular, $\Phi_{E_k}\to\Phi_E$.
		\end{lemma}
		
		The proof of Lemma~\ref{lemma::phi_continuous_in_E} is a variation of the proof of~\cite[Proposition~9.1]{MR2675483}. See also \cite[Appendix H]{MR5015001} for a proof in the setting of $(\Lambda,0)$-minimal sets.
		
		We point out that~$E$ is a minimal cone with vertex at the origin, that is the~$(n+1)$-dimensional extension of~$\chi_E-\chi_{E^c}$ is a~$0$-homogeneous function, if and only if the function~$\Xi_E$, as defined in~\eqref{eq::def_Xi}, is constant (see~\cite[Corollary~8.2]{MR2675483}).
		
		By Lemma~\ref{lemma::phi_continuous_in_E}, we infer that the limit set of a convergent blow-up sequence~$\{E_k:=\lambda_k E\}_k$ is a minimal cone, according to the following:
		
		\begin{theorem}[Blow up limit] \label{th::blowup_cone}
			Let~$E$ be a~$(\Lambda,\gamma)$-almost minimal set in~$B_1$ such that~$0\in\partial E$, and let~$\{\lambda_k\}_k$ be a sequence of positive real numbers
			such that~$\lambda_k\to+\infty$ as~$k\to+\infty$. 
			
			Then, there exists a set~$C$ such that~$\lambda_k E\to C$ in~$L^1_{\loc}(\R^n)$ (up to a subsequence). Moreover,~$C$ is a minimal cone tangent to~$E$ at the origin.
		\end{theorem}
		
		\begin{proof}
			Notice that if~$E$ is a~$(\Lambda,\gamma)$-almost minimal set in~$B_1$, then~$\lambda_k E$ is a~$(\lambda_k^{-\gamma}\Lambda,\gamma)$-almost minimal set in~$B_{\lambda_k}$. Therefore, for~$k$ sufficiently large, we have that~$\lambda_k E$ is almost minimal also in~$B_1$.
			
			Let us consider the set~$A_{k}^1:=B_1\cup(\lambda_k E\cap B_1^c)$. In this setting, we have that~$A_{k}^1\Delta\lambda_k E\subseteq B_1$, and, by the almost minimality of~$\lambda_k E$,
			\begin{align*}
				&\norm{\chi_{\lambda_k E}}_{W^{1,s}(B_1)} = \norm{\chi_{\lambda_k E}}_{L^1(B_1)}+[\chi_{\lambda_k E}]_{W^{1,s}(B_1)}\\
				&\qquad\leq |B_1|+2\Per_s(\lambda_k E,B_1)\leq |B_1|+2\Per_s(A_{k}^1,B_1) + 2\lambda_k^{-\gamma}\Lambda\\
				&\qquad\leq |B_1|+2\Per_s(B_1,\R^n)  +2\Lambda,
			\end{align*} 
			provided that~$k$ is sufficiently large.
			
			Using the Sobolev embedding~$W^{1,s}(B_1)\hookrightarrow\hookrightarrow L^1(B_1)$, we extract a subsequence~$\lambda_k^1$ such that
			$$ (\lambda_k^1 E)\cap B_1 \to E_1 \text{ in } L^1(B_1)$$
			for some set~$E_1\subseteq B_1$. Repeating the argument with~$\{\lambda_k^1 E\}_k$ in~$B_2$, we find a subsequence~$\{\lambda_k^2\}_k\subseteq\{\lambda_k^1\}_k$ and a set~$E_2\subseteq B_2$ such  that 
			$$ (\lambda_k^2 E)\cap B_2 \to E_2 \text{ in } L^1(B_2).$$
			Moreover, by the uniqueness of the limit, we have that~$E_2\cap B_1=E_1$.
			
			Iterating this procedure~$m$ times, we obtain a subsequence~$\{\lambda_k^m\}_k$ such that~$ (\lambda_k^m E)\cap B_m \to E_m$ in~$L^1(B_m)$ as~$k\to+\infty$, and~$E_m\cap B_j=E_j$, for all~$j=1,\dots,m-1$.
			
			Using a diagonal argument, we extract a subsequence~$\{\lambda_k^k\}_k$ such that
			\begin{equation} \label{eq::conv_to_cone}
				\lambda_k^k E \to \bigcup_{m=1}^{+\infty} E_m=:C \text{ in }L^1_{\loc}(\R^n).
			\end{equation}
			
			Furthermore, recall that~$\lambda_k E$ is a~$(\lambda_k^{-\gamma}\Lambda,\gamma)$-almost minimal set
			in~$B_{\lambda_k}$. Thus, by~\eqref{eq::conv_to_cone} and Lemma~\ref{lemma::phi_continuous_in_E}, we infer that
			$$\lim_{k\to+\infty} \Phi_{\lambda_k E}(r)=\Phi_C(r).$$
			
			As a consequence, recalling the notation in~\eqref{eq::def_Xi}, we obtain that
			\begin{equation*}
				\begin{split}
					&\Phi_C(r) =\lim_{k\to+\infty}\Phi_{\lambda_k E}(r) 
					= \lim_{k\to+\infty}\Xi_{\lambda_k E}(r) + \frac{n-s}{\gamma}\hat{\Lambda} r^{\gamma}\\
					&\qquad\qquad=\lim_{k\to+\infty} \Xi_E\left(\frac{r}{\lambda_k }\right) +\frac{n-s}{\gamma}\hat{\Lambda} r^{\gamma}.
				\end{split}
			\end{equation*}
			
			Now, since~$\Phi_E$ is positive and non-decreasing, we have that~$\Phi_E(r/\lambda_k)$ admits limit as~$k\to+\infty$ Therefore, also recalling that $\gamma<s$, we infer that
			$$\lim_{k\to+\infty} \Xi_E\left(\frac{r}{\lambda_k}\right) =\lim_{k\to+\infty} \Phi_E\left(\frac{r}{\lambda_k}\right) - \frac{n-s}{\gamma}\hat{\Lambda} \left(\frac{r}{\lambda_k}\right)^{\gamma} = \Phi_E(0) = \Xi_E(0) .$$
			As a result, we find that
			$$	\Phi_C(r)=	\Xi_E(0)+\frac{n-s}{\gamma}\hat{\Lambda} r^{\gamma},$$
			which reads
			$$\Xi_C(r)	= \Phi_C(r) - \frac{n-s}{\gamma}\hat{\Lambda} r^{\gamma} =\Xi_E(0).$$
			This shows that~$\Xi_C$ is constant, hence~$C$ is a minimal cone with vertex at~$0$.
			
		\end{proof}
		
		\subsection{Classification of tangent cones}
		Now, we show how the classification of the tangent cones to an almost minimal set allows to deduce regularity.
		
		\begin{theorem}\label{th::regular_tangent_cone}
			Let~$E$ be a~$(\Lambda,\gamma)$-almost minimal set in~$B_1$ such that~$0\in\partial E$. Suppose that the tangent cone to~$E$ at~$0$ is a half-space. 
			
			Then, there exists~$\delta>0$ such that~$\partial E\cap B_\delta$ is a~$\cont^{1,\alpha}$-regular surface.
		\end{theorem}
		
		The proof of Theorem~\ref{th::regular_tangent_cone} relies on following auxiliary result and will be presented here below. 
		
		\begin{theorem}[Energy gap]\label{th::energy_gap}
			Let~$\Pi:=\{x_1>0\}$. 
			Then, for every tangent cone~$C$, 
			\begin{equation}\label{eq::PileqC}
				\Xi_\Pi\leq\Xi_C,
			\end{equation}
			and equality holds if and only if~$C$ is a half-space. 
			
			In particular, if~$C$ is not a half-space, then
			there exists~$\delta_0>0$, depending only on~$n$ and~$s$, such that
			\begin{equation}\label{eq::gap}
				\Xi_\Pi\leq\Xi_C-\delta_0 .
			\end{equation}
		\end{theorem}
		
		In other words, Theorem~\ref{th::energy_gap} states that half-spaces are the tangent cones with least local energy. Moreover, there exists an energy gap between regular and singular tangent cones. 
		We remark that the same argument as the one employed to prove ~\cite[Theorem~9.6]{MR2675483} shows Theorem~\ref{th::energy_gap}, up to using the clean ball condition for almost minimizers (see Corollary~\ref{cor::clean_ball}). 
		
		\begin{proof}[Proof of Theorem~\ref{th::regular_tangent_cone}]
			Let $\{\lambda_k\}_k$ be a sequence of positive real  numbers such that $\lambda_k\to+\infty$, as $k\to+\infty$. Let us consider the blow-up sequence~$\{\lambda_k E\}_k$. Then, by assumption, we have that~$\lambda_k E\to \Pi$ in~$L^1_{\loc}$, for some half-space~$\Pi$. By Corollary~\ref{cor::boundary_conv}, we deduce that $\lambda_k E\to \Pi$ also uniformly.
			
			Moreover, up to rotations, we can always assume that~$\Pi=\{x_n<0\}$. 	
			Thus, there exists some~$\overline{k}\in\N$ such that for every~$k\geq\overline{k}$
			$$ \partial(\lambda_k E)\cap B_1 \subseteq \{|x_n|<\eta_0\}, $$
			where $\eta_0$ is as in Theorem~\ref{th::holder_reg_almost_min}.
			
			In particular, thanks to Theorem~\ref{th::holder_reg_almost_min}, we infer that ~$\partial(\lambda_{\overline{k}} E)$ is a~$\cont^{1,\alpha}$-regular surface in~$B_{1/2}$, and hence that~$\partial E$ is a~$\cont^{1,\alpha}$-regular surface in~$B_{1/(2\lambda_{\overline{k}})}$.
		\end{proof}

		
		\appendix
		
		\section{Quasi-minimality implies almost minimality (Lemma~\ref{lemma::quasi_are_almost})} \label{sec::quasi_are_almost}
		
		This section is devoted to proving Lemma~\ref{lemma::quasi_are_almost}. 
		
		First, observe that, if $E$ is a $(K,\beta)$-quasi minimal set in $B_1$, then, for any ball $\widetilde{B}_r \Subset B_1$,
		\begin{equation} \label{eq::per_bound_quasi_min}
			\Per_s(E,\widetilde{B}_r) \leq \widetilde{K}r^{n-s},
		\end{equation}
		for some $\widetilde{K}>0$ depending only on $n$ and $s$. Indeed, using quasi minimality with respect to the competitor $F_r:=E\setminus \widetilde{B}_r$, we find
		\begin{equation*}
			\Per_s(E,\widetilde{B}_r) \leq (1+Kr^\beta) \Per_s(F_r,\widetilde{B}_r) \leq \widetilde{K}\mathcal{L}_s(\widetilde{B}_r,\widetilde{B}_r^c),
		\end{equation*}
		up to renaming $\widetilde{K}$.
		
		Let now $F\subseteq\R^n$ be any set such that $E\Delta F\subseteq \widetilde{B}_r$, for some ball $\widetilde{B}_r \Subset B_1$. Therefore, by quasi minimality, we have
		\begin{equation*}
			\frac{1}{1+Kr^\beta}\Per_s(E,\widetilde{B}_r) \leq \Per_s(F,\widetilde{B}_r) 
		\end{equation*}
		which, together with~\eqref{eq::per_bound_quasi_min},reads
		\begin{equation*}
			\Per_s(E,\widetilde{B}_r) \leq \Per_s(F,\widetilde{B}_r) + Kr^\beta \Per_s(E,\widetilde{B}_r) \leq \Per_s(F,\widetilde{B}_r) + \widetilde{K}r^{n-s+\beta},
		\end{equation*}
		establishing the desired result.		
		
		
		\section{Proof of the H\"older estimates (Lemma~\ref{lemma::holder_est})} \label{sec::holder_est_proof}
		
		In this section, we present a proof of Lemma~\ref{lemma::holder_est}. We refer to \cite[Lemma 2.20]{MR5015001} for more details (when $\gamma=0$). However, we stress that the argument does not depend explicitly on $\gamma$.
		
		Let~$k\in\N$, and take~$x_0\in A_k^r\cap\{|x'|<r/2\}$. Define 
		$$ \Sigma_{k,j}:=A_k^r \cap \left\{|x'-x_0'|<\frac{1}{2}\left(\frac{\delta_0}{2}\right)^j\right\}.$$
		Notice that, since 
		\begin{equation*}
			\partial E^*_k \cap C_{2^h}\subseteq \big\{|x\cdot\nu_h^k|\leq 2^{h(1+\alpha)}\big\},\quad{\mbox{for all }}h\in\{ 0,\dots,h_k\},
		\end{equation*}
		then also
		\begin{equation*}
			\Sigma_{k,j} = \Sigma_{k,j}\cap C_{2^h} \subseteq \big\{|x\cdot\nu_h^k|\leq 2^{h(1+\alpha)}\big\},\quad{\mbox{for all }}h\in\{ 0,\dots,h_k\}.
		\end{equation*}

		Therefore, applying iteratively Harnack's Inequality in Theorem~\ref{th::harnack_ineq} (see \cite[Corollary 2.17]{MR5015001}), we infer that
		\begin{equation} \label{eq::Sigma_trap}
			\Sigma_{k,j} \subseteq \left\{|x_n-z_j|< 2^{2+\alpha}\left(1-\frac{\delta_0^2}{2}\right)^j\right\},\qquad{\mbox{for all }}j\in\{ 0,\dots, j_k\} ,
		\end{equation}
		for some $z_j=z_j(x_0)$, with $|z_j|\leq1$ and 
		\begin{equation} \label{eq::z_j_nested}
			z_j-\left(1-\frac{\delta_0^2}{2}\right)^j \leq z_{j+1}-\left(1-\frac{\delta_0^2}{2}\right)^j \leq  z_{j+1}+\left(1-\frac{\delta_0^2}{2}\right)^j \leq  z_{j}+\left(1-\frac{\delta_0^2}{2}\right)^j ,
		\end{equation} 
		for all $j$ (and all $k$).
		Notice also that~$j_k\to+\infty$ as~$k\to+\infty$. 
		
		We claim that, for every~$x\in A_k^r\cap\{|x'|<r/2\}$,
		\begin{equation}\label{261BIS00}
			|x_n-z_{j_k}|<\theta\max\{|x'-x_0'|,b_k\}^\eta ,
		\end{equation}
		where~$b_k$ is an infinitesimal sequence of positive real numbers.
		
		Let us first observe that, given~$x \in A_k^r\cap\{|x'|<r/2\}$, we have that~$|x'-x_0'|<r$. Then, we distinguish three cases depending on whether~$|x'-x_0'|$ is in between~$r$ and~$\frac{1}{2}$, in between~$\frac{1}{2}$ and~$\frac{1}{2}\left(\frac{\delta_0}{2}\right)^{j_k}$, or smaller than~$\frac{1}{2}\left(\frac{\delta_0}{2}\right)^{j_k}$. 
		
		Let us start from the case~$\frac{1}{2}\left(\frac{\delta_0}{2}\right)^{j_k}<|x'-x_0'|<\frac{1}{2}$. In this situation, there exists~$j\leq j_k$ in~$\N$ such that
		$$ \frac{1}{2}\left(\frac{\delta_0}{2}\right)^{j+1}\leq|x'-x_0'|<\frac{1}{2}\left(\frac{\delta_0}{2}\right)^{j}.$$
		Rearranging the latter, we obtain
		$$ j< \frac{\log(2|x'-x_0'|)}{\log(\delta_0/2)}\leq j+1.$$
		Moreover,~\eqref{eq::Sigma_trap} yields~$|x_n-z_j|\leq 2^{2+\alpha} \left( 1-\frac{\delta_0^2}{2}\right)^j$.
		Therefore, setting
		\begin{equation}\label{yrueiow5674836547839gfhdjvfbdns}
			\eta:=\frac{\log(1-\frac{\delta_0^2}{2})}{\log(\delta_0/2)}\in(0,1),
		\end{equation} 
		we obtain that
		\begin{align*}
			\left(1-\frac{\delta_0^2}{2}\right)^j 
			\leq 
			\left(1-\frac{\delta_0^2}{2}\right)^{\eta\frac{\log(2|x'-x_0'|)}{\log(1-\delta_0^2/2)}-1}
			=\left(1-\frac{\delta_0^2}{2}\right)^{-1}2^\eta|x'-x_0'|^\eta . 
		\end{align*}
		Thus, defining~$\theta:=2^\eta(1+2^{2+\alpha}) \left(1-\frac{\delta_0^2}{2}\right)^{-1}$ and also recalling \eqref{eq::z_j_nested}, we have that
		\begin{equation*} 
			|x_n-z_{j_k}| \leq  |x_n-z_j|+|z_j-z_{j_k}| \leq (1+2^{2+\alpha})\left(1-\frac{\delta_0^2}{2}\right)^j \leq \theta |x'-x_0'|^\eta,
		\end{equation*}
		which gives~\eqref{261BIS00} in this case.
		
		Now, if~$\frac{1}{2}\leq|x'-x_0'|<r$, we have that~$|x'-x_0'|^\eta\geq2^{-\eta}$, where~$\eta$ is as in~\eqref{yrueiow5674836547839gfhdjvfbdns}. Also,
		$$ |x_n-z_{j_k}|\leq |x_n|+|z_{j_k}| \leq 2C.$$
		Thus,
		$$ |x_n-z_{j_k}|\leq  \theta |x'-x_0'|^\eta,$$
		provided that~$\theta\geq 2^{1+\eta}C$.
		
		As last, we deal with the case~$|x'-x_0'|\leq \frac{1}{2}\left(\frac{\delta_0}{2}\right)^{j_k}$. Since $\Sigma_{k,j_k} \subseteq  \Sigma_{k,j_k-1}$, recalling \eqref{eq::z_j_nested} once again, we deduce that
		\begin{equation*}
			|x_n-z_{j_k}|<2^{2+\alpha}\left(1-\frac{\delta_0^2}{2}\right)^{j_k-1} + |z_{j_k-1}-z_{j_k}| \leq (1+2^{2+\alpha})\left(1-\frac{\delta_0^2}{2}\right)^{j_k-1}.
		\end{equation*}
		Defining $b_k:=\left(1-\frac{\delta_0^2}{2}\right)^{j_k}$, we obtain $b_k\to0$ as~$k\to+\infty$, and
		\begin{equation*}
			|x_n-z_{j_k}|< (1+2^{2+\alpha})\left(1-\frac{\delta_0^2}{2}\right)^{-1}b_k \leq\theta b_k^\eta,
		\end{equation*}
		which completes the proof of~\eqref{261BIS00}.
		
		From \eqref{261BIS00}, we infer that~$A_k^r\cap \left\{|x'-x_0'|<\frac{r}{2}\right\} $ lies above the graph of the H\"older continuous function 
		$$ f_k^-[x_0](x') := z_{j_k}- \theta\max\{b_k,|x'-x_0'|\}^{\eta},$$
		and below the graph of the H\"older continuous function 
		$$ f_k^+[x_0](x') := z_{j_k}+ \theta\max\{b_k,|x'-x_0'|\}^{\eta} .$$
		Taking the infimum and the supremum respectively over all~$x_0\in A_k^r\cap \left\{|x'-x_0'|<\frac{r}{2}\right\} $, we conclude that~$A_k^r\cap \left\{|x'-x_0'|<\frac{r}{2}\right\} $ lies between the graphs of the H\"older continuous functions
		\begin{align*}
			&g_k^-:=\inf_{x_0\in A_k^r\cap \left\{|x'-x_0'|<\frac{r}{2}\right\} } f_k^-[x_0]\\
			\text{and }\qquad &g_k^+:=\sup_{x_0\in A_k^r\cap \left\{|x'-x_0'|<\frac{r}{2}\right\} } f_k^+[x_0], 
		\end{align*}
		concluding the proof.
		
		
		\section{Monotonicity formula (Proposition~\ref{prop::Phi_monotone})} \label{sec::monotonicity}
		
		In this section, we prove the monotonicity formula stated in Proposition~\ref{prop::Phi_monotone}. To this purpose, we recall the extension technique presented in~\cite{MR2354493}.
		Let
		$$\R^{n+1}_+:=\big\{(x,z)\in\R^{n+1}\;{\mbox{ s.t. }}\; x\in\R^n,\ z\in[0,+\infty)\big\} .$$ 
		the~$(n+1)$-dimensional half-space. Given a bounded Lipschitz open domain~$\Omega\subseteq\R^{n+1}$, let also
		$$ \Omega_0:=\Omega\cap\{z=0\} \qquad{\mbox{and}}\qquad \Omega_+:=\Omega\cap\{z>0\} .$$
		
		For any set~$E\subseteq\R^n\times\{0\}$, we define the~$(n+1)$-dimensional extension of~$u:=\chi_E-\chi_{E^c}$ as the convolution
		\begin{equation}\label{TILDEUDEF}
			\widetilde{u}(x,z):=u*P(x,z),
		\end{equation}
		where
		\begin{align*}
			&P(x,z):=C_P \frac{z^s}{(x^2+z^2)^{\frac{n+s}{2}}},\\
			\text{and}\quad&C_P:=\left(\int_{\R^n}\frac{dx}{(1+x^2)^{(n+2)/2}}\right)^{-1}.
		\end{align*}
		
		Then, we have the following equivalence:
		
		\begin{prop}[{\cite[Proposition 1.7]{MR5015001}}] \label{prop::char_almost_min_ext}
			$E$ is a~$(\Lambda,\gamma)$-almost minimal set in~$B_1$ if and only if there exists a positive constant~$\widetilde{c}_{n,s}$ such that the extension~$\widetilde{u}$ of~$u=\chi_E-\chi_{E^c}$ to~$\R^{n+1}_+$ satisfies
			\begin{equation} \label{eq::char_almos_min_ext}
				\int_{\Omega_+} z^{1-s}|\nabla\widetilde{u}|^2\,dx\,dz \leq \int_{(\widetilde{B}_r^{n+1})_+} z^{1-s}|\nabla\overline{v}|^2\,dx\,dz + 8\widetilde{c}_{n,s}\Lambda r^{n-s+\gamma}
			\end{equation}
			for every $\Omega \subseteq \R^{n+1}$ such that~$\Omega_0 \subseteq  \widetilde{B}_r\times\{0\} \Subset B_1\times\{0\}$, for some ball of radius $r>0$, and for every function~$\overline{v}:\R^{n+1}_+\to\R$ such that $\mbox{supp}(\overline{v}-\widetilde{u}) \subseteq \Omega$ and~$|\overline{v}|=1$ on~$\Omega_0$.
		\end{prop}
		
		In other words, Proposition~\ref{prop::char_almost_min_ext} states that a set $E$ is almost minimal if and only if its extension $\tilde{u}$ is an almost minimizer for a suitable energy.
		
		\begin{rem} \label{rem::hat_Lambda}
			We point out that the constant~$\widetilde{c}_{n,s}$ in Proposition~\ref{prop::char_almost_min_ext} here is the same as in~\cite[Lemma~7.2]{MR2675483}. 
			Since it will appear several times in the forthcoming discussion, given~$\Lambda\geq0$, we define the constant
			$$ \hat{\Lambda} := 8\widetilde{c}_{n,s}\Lambda.$$
		\end{rem}
		
		Now, we consider a set~$E\subseteq\R^n$ and~$R>0$. Recall from~\eqref{eq::defphimono} that for every~$r\in(0,R)$, 
		\begin{equation} 
			\Phi_E(r) := r^{s-n}\int_{B_r^+} z^{1-s}|\nabla\widetilde{u}|^2\,dx\,dz + \frac{n-s}{\gamma}\hat{\Lambda}r^{\gamma}.
		\end{equation}
		For later convenience, we also define
		\begin{equation} \label{eq::def_Xi}
			\Xi_E(r) :=r^{s-n}\int_{B_r^+} z^{1-s}|\nabla\widetilde{u}|^2\,dx\,dz.
		\end{equation}
		
		The proof of Proposition~\ref{prop::Phi_monotone} is a modification of the one for the case $\gamma=s$, which can be found and is contained in \cite[Section 3]{MR5015001}(as well as a proof of the analogue of Proposition~\ref{prop::char_almost_min_ext}). For the convenience of the reader, we sketch the main steps below. The result is a consequence of the following intermediate results.
		
		\begin{prop} \label{prop::Phi_abs_cont}
			Let $0<r_1<r_2<R$. Then, the function $\Phi_E$ is absolutely continuous in $[r_1,r_2]$.
		\end{prop}
		\begin{proof}
			It follows immediately from the absolute continuity of the Lebesgue integral and the smoothness of $r^{s-\gamma}$. \qedhere
		\end{proof}
		
		\begin{prop} \label{prop::Phi_monotone_ae}
			The function $\Phi_E$ is differentiable for almost every $r\in(0,R)$, and
			\begin{equation*} \label{eq::Phi_monotone_ae}
				\Phi_E'(r)\geq0,\quad\mbox{for a.e. }r\in(0,R).
			\end{equation*}
		\end{prop}
		
		\begin{proof}
			Thanks to Proposition~\ref{prop::Phi_abs_cont} and the Fundamental Theorem of Calculus for the Lebesgue integral (see e.g. \cite[Theorem 7.20]{rudin}), $\Phi_E$ is differentiable for almost any $r\in(0,R)$ with derivative
			\begin{align*}
				\Phi_E'(r) 
				&= r^{s-n} \int_{\partial B_r^+} z^{1-s}|\nabla\widetilde{u}(x,z)|^2\, d\haus{n}(x,z) \\
				&\qquad\qquad-(n-s)r^{s-n-1} \int_{B_r^+} z^{1-s}|\nabla\widetilde{u}(x,z)|^2 \,dx\,dz +(n-s)\,\hat{\Lambda} r^{\gamma-1} \\
				&= (n-s)r^{s-n-1} \Bigg(\frac{r}{n-s}\int_{\partial B_r^+} z^{1-s}|\nabla\widetilde{u}(x,z)|^2\, d\haus{n}(x,z) \\	
				&\qquad\qquad - \int_{B_r^+} z^{1-s}|\nabla\widetilde{u}(x,z)|^2 \,dx\,dz +\hat{\Lambda} r^{n-s+\gamma}\Bigg).
			\end{align*}
			Therefore, the positivity of $\Phi_E'$ will follow, if we show that
			\begin{equation} \label{eq::phi_monotone__equiv}
				\frac{r}{n-s}\int_{\partial B_r^+} z^{1-s}|\nabla\widetilde{u}(x,z)|^2\, d\haus{n}(x,z) - \int_{B_r^+} z^{1-s}|\nabla\widetilde{u}(x,z)|^2 \,dx\,dz + \hat{\Lambda} r^{n-s+\gamma-1}\geq0.
			\end{equation}
			
			With this aim, we consider the function
			\begin{equation*}
				\widetilde{v}_r(x,z):=
				\begin{cases}
					\widetilde{u}\left(\frac{rx}{|(x,z)|},\frac{rz}{|(x,z)|}\right) ,&\quad \text{  if }0<|x|^2+z^2\leq r^2,\\
					\widetilde{u}(x,z) ,&\quad \text{ if }|x|^2+z^2> r^2.
				\end{cases}
			\end{equation*}
			
			We claim that there exists a sets $\mathcal{W}\subseteq(0,+\infty)$, with $\haus{1}\left((0,+\infty)\setminus\mathcal{W}\right)=0$, such that for all $r\in{\mathcal{W}}$ there exists $\mathcal{G}_r\subseteq\R^n$ of full measure in $\R^n$ such that if~$x\in \mathcal{G}_r$, then 
			\begin{equation*}
				\widetilde{v}_r(x,0)=
				\begin{cases}
					\widetilde{u}\left(\frac{rx}{|x|},0\right)=\chi_E\left(\frac{rx}{|x|}\right)-\chi_{E^c}\left(\frac{rx}{|x|}\right),\quad&\text{if }0<|x|\leq r,\\
					\widetilde{u}(x,0)=\chi_E(x)-\chi_{E^c}(x),\quad&\text{if }|x|> r,
				\end{cases}
			\end{equation*}
			is well-defined (and thus $\widetilde{v}_r(x,0)\in\{-1,1\}$). The existence of $\mathcal{W}$ and $\mathcal{G}_r$ can be deduced with a careful measure-theoretic analysis of the set where the trace of $\widetilde{v}_r$ is defined. For a more detailed discussion, we refer to the proof of Lemma 3.3 in \cite{MR5015001}.
			
			In what follows, we pick $r\in(0,R)\cap{\mathcal{W}}$, and we write $\widetilde{v}:=\widetilde{v}_r$ and $F:=F_r$ for simplicity.
			Then, thanks Proposition~\ref{prop::char_almost_min_ext}, we infer that 
			\begin{equation}\label{eq::monotone_1}
				\int_{B_r^+} z^{1-s}|\nabla\widetilde{v}(x,z)|^2\,dx\,dz
					\geq \int_{B_r^+} z^{1-s}|\nabla\widetilde{u}(x,z)|^2\,dx\,dz - \hat{\Lambda}r^{n-s+\gamma}
			\end{equation}
			
			Since~$\widetilde{v}$ is constant along the radial direction in $B_r\setminus\{0\}$ by construction, we have that
			\begin{equation*}
				\nabla\widetilde{v}(x,z) = \frac{r}{|(x,z)|} \nabla^T \widetilde{u}\left(\frac{r}{|(x,z)|}x,\frac{r}{|(x,z)|}z\right),\quad \mbox{in }B_r\setminus\{0\},
			\end{equation*}
			where~$\nabla^T$ is the tangential component of the gradient~$\nabla$ with respect to the unitary sphere. 
			
			Therefore, by mean of Fubini-Tonelli's Theorem, we deduce that
			\begin{equation*}
				\begin{split}
					\int_{B_r^+} z^{1-s}|\nabla\widetilde{v}(x,z)|^2\,dx\,dz 
					&= \int_{0}^{r}\int_{\partial B_t^+} z^{1-s}|\nabla\widetilde{v}(x,z)|^2\,d\haus{n}(x,z)\, dt \\
					&= \int_{0}^{r}\int_{\partial B_t^+} \frac{r^2}{t^2}z^{1-s}\left|\nabla^T\widetilde{u}\left(\frac{r}{t}x,\frac{r}{t}z\right)\right|^2\,d\haus{n}(x,z)\, dt.
				\end{split}
			\end{equation*}
			Adopting the change of variable~$(\hat{x},\hat{z}):=\left(\frac{r}{t}x,\frac{r}{t}z\right)$, we obtain that
			\begin{equation} \label{eq::monotone_2}
				\begin{split}
					\int_{B_r^+} z^{1-s}|\nabla\widetilde{v}(x,z)|^2\,dx\,dz 
					& = \int_{0}^{r}\int_{\partial B_r^+} \left(\frac{t}{r}\right)^{n-2}\left(\frac{t}{r}\hat{z}\right)^{1-s}|\nabla^T\widetilde{u}(\hat{x},\hat{z})|^2\,d\haus{n}(\hat{x},\hat{z})\, dt\\
					& = \int_{0}^{r} \left(\frac{t}{r}\right)^{n-1-s}\,dt \int_{\partial B_r^+} \hat{z}^{1-s}|\nabla^T\widetilde{u}(\hat{x},\hat{z})|^2\,d\haus{n}(\hat{x},\hat{z}) \\
					& = \frac{r}{n-s} \int_{\partial B_r^+} \hat{z}^{1-s}|\nabla^T\widetilde{u}(\hat{x},\hat{z})|^2\,d\haus{n}(\hat{x},\hat{z}).
				\end{split}
			\end{equation}
			
			Plugging~\eqref{eq::monotone_2} into~\eqref{eq::monotone_1}, we obtain that
			\begin{equation*}
				\begin{split}
					\frac{r}{n-s} \int_{\partial B_r^+} z^{1-s}|\nabla \widetilde{u}(x,z)|^2\,d\haus{n}(x,z)
					&\geq \frac{r}{n-s} \int_{\partial B_r^+} z^{1-s}|\nabla^T\widetilde{u}(x,z)|^2\,d\haus{n}(x,z)\\
					&\geq \int_{B_r^+} z^{1-s}|\nabla \widetilde{u}|^2\,dx\,dz - \hat{\Lambda} r^{n-s+\gamma}\  ,
				\end{split}
			\end{equation*}
			which gives~\eqref{eq::phi_monotone__equiv}, as desired. 
		\end{proof}

		
		\section{A note on $Q$-quasi minimality} \label{sec::Q_min}
		
		In this section we focus on $Q$-quasi minimal sets, for some $Q>1$. In particular, we prove that $Q$-quasi minimizers for the $s$-Perimeter still satisfy uniform density estimate. However, in general, they are no better than Lipschitz, showing that the assumption $\gamma>0$ in Theorem~\ref{th::holder_reg_almost_min} is necessary.

		\begin{prop}[Uniform density estimates] \label{prop::UDE_Q_min} 
			Let~$E$ be a~$Q$-quasi minimal set in $B_1$, for some $Q>1$.
			Then, there exist constants~$r_0>0$ and~$c_0\in(0,1)$, depending only on~$n$, $s$,~$\Lambda$, and~$\gamma$, such that, for any~$x_0\in(\partial E)\cap B_1$ and~$r\in(0,\min\{r_0,\mbox{dist}(x_0,\partial B_1)\})$,
			\begin{equation}\label{eq::unif_dens_estimates_Q_min}
				\begin{split}
					&|E\cap B_r(x_0)|\geq c_0 r^n\\
					\mbox{and}\quad&|E^c\cap B_r(x_0)|\geq c_0 r^n.
				\end{split}
			\end{equation}
		\end{prop}
		
		We recall that, , for any set $E$, the density of $E$ in the ball of radius $r$ as
		\begin{equation*}
			a(r):=\frac{|E\cap B_r|}{|B_r|}.
		\end{equation*}
		
		The proof of Proposition~\ref{prop::UDE_Q_min} follows the same structure as the one of Proposition~\ref{prop::UDE}. However, in the setting of $Q$-quasi minimizers, the argument is more delicate as it requires estimating further tail terms. For this reason, and for the sake of clarity, we provide the full details here.
		
		\begin{lemma} \label{lemma::keep_small_density_Q_min}
			Let $\delta_0>0$.
			Let~$E$ be a~$Q$-quasi minimal set in $B_{2^{\overline{k}}}$, with $\overline{k}\in\N$ such that $2^{-\overline{k}s}\leq \delta_0$. If
			\begin{equation} \label{eq::small_density_dyadic}
				a(2^{k})\leq \delta_0, \quad\text{for all }k\in\{0,\dots,\overline{k}\},
			\end{equation}
			then $a(1/2)\leq \delta_0$. 
		\end{lemma}
		
		\begin{proof}
			Let $r\in\left(\frac{1}{2},1\right)$, and define $A_r:=E\cap B_r$. Also, let~$\mu(r):=|A_r|$ and notice that, by the co-area formula, $\mu'(r)=\haus{n-1}(E\cap\partial B_r)$.
			
			Now, notice that, by Sobolev embeddings, we have that  
			\begin{equation} \label{eq::ude1_Q_min}
				\norm{\chi_{A_r}}_{L^{\frac{n}{n-s}}(\R^n)} = |A_r|^{1-\frac{s}{n}}\leq C[\chi_{A_r}]_{H^{s/2}} = C \mathcal{L}_s(A_r,A_r^c),
			\end{equation}
			for some constant~$C>0$ depending only on~$n$ and ~$s$.
			
			Since
			\begin{equation*}
				\mathcal{L}_s(A_r,A_r^c) = \mathcal{L}_s(A_r,E\cap A_r^c)+\mathcal{L}_s(A_r,E^c),
			\end{equation*}
			it follows from the $Q$-quasi minimality of $E$ that
			\begin{equation} \label{eq::ude2_Q_min}
				\begin{split}
					&\mathcal{L}_s(A_r,A_r^c)\leq 2\mathcal{L}_s(A_r,E\cap A_r^c) + Q\Per_s(E\setminus B_r,B_r)\\
					&\qquad\leq 2\mathcal{L}_s(A_r,B_r^c) + Q\mathcal{L}_s(E\cap B_r^c,B_r).
				\end{split}
			\end{equation}
			From~\eqref{eq::ude1_Q_min} and~\eqref{eq::ude2_Q_min} together, we deduce that
			\begin{equation} \label{eq::ude3_Q_min}
				\mu(r)^{1-\frac{s}{n}} \leq C\mathcal{L}_s(A_r,B_r^c) + Q\mathcal{L}_s(E\cap B_r^c,B_r).
			\end{equation}
			and we aim to estimate the two terms on the right-hand side.
			
			To do so, observe that, by the co-area formula, we have that
			\begin{equation} \label{eq::ude4_Q_min}
				\begin{split}
					&	\mathcal{L}_s(A_r,B_r^c)
					= \iint_{A_r\times B_r^c} \frac{dx\,dy}{|x-y|^{n+s}}
					\leq C\int_{A_r}\left(\int_{r-|x|}^{+\infty}\frac{dz}{z^{s+1}}\right)\,dx \\
					&\qquad\quad\leq C\int_{A_r} \frac{dx}{(r-|x|)^{s}}
					\leq C\int_0^r\frac{\mu'(\rho)}{(r-\rho)^{s}}\, d\rho ,
				\end{split}
			\end{equation}
			for some~$C>0$, depending on~$n$, and~$s$ and possibly changing from line to line.

			Then, we provide an estimate for
			\begin{equation*}
				\begin{split}
					&\mathcal{L}_s(E\cap B_r^c,B_r)\\
					&\quad = \iint_{B_r\times E\cap (B_1\setminus B_r)} \frac{dxdy}{|x-y|^{n+s}} +\sum_{k=0}^N \iint_{B_r\times E\cap C_k} \frac{dxdy}{|x-y|^{n+s}} + \iint_{B_r\times E\cap B_{2^N}^c}\frac{dxdy}{|x-y|^{n+s}}
				\end{split}
			\end{equation*}
			where $C_k:=B_{2^{k+1}}\setminus B_{2^k}$.
			
			First, using co-area formula, we deduce that
			\begin{equation} \label{eq::ude5_1_Q_min}
				\begin{split}
					&\iint_{B_r\times E\cap (B_1\setminus B_r)} \frac{dxdy}{|x-y|^{n+s}} \leq  \int_{E\cap (B_1\setminus B_r)}\int_{B_{|x|}(x)\setminus B_{|x|-r}(x)} \frac{dxdy}{|x-y|^{n+s}} \\
					&\qquad\leq C  \int_{E\cap (B_1\setminus B_r)} \int_{|x|-r}^{|x|}\frac{dxdz}{|z|^{n+s}}\leq C  \int_{E\cap (B_1\setminus B_r)}\frac{dx}{(|x|-r)^{s}}\\
					&\qquad = C \int_{r}^1 \frac{\mu'(\rho)}{(\rho-r)^s}d\rho .
				\end{split}
			\end{equation}
			
			Moreover, for every $k\in\{0,\dots,\overline{k}-1\}$, we have that
			\begin{equation} \label{eq::ude5_2_Q_min}
				\begin{split}
					& \iint_{B_r\times E\cap C_k} \frac{dxdy}{|x-y|^{n+s}} \leq \int_{B_r \times E\cap C_k} \frac{dxdy}{(2^k-r)^{n+s}} \\
					&\qquad\leq C 2^{-k(n+s)} |E\cap C_k| \leq C2^{-ks} \frac{|E\cap B_{2^{k+1}}|}{2^{n(k+1)}}\leq C2^{-ks} a(2^{k+1}), 
				\end{split}
			\end{equation}
			for some constant $C>0$ depending only on $n$ and $s$, and possibly changing from line to line.
			
			As last, since $|x-y|\geq \frac{|x|}{2}$, whenever $x\in E\cap B_{2^{\overline{k}}}^c$ and $y\in B_r$, we infer that
			\begin{equation} \label{eq::ude5_3_Q_min}
				\begin{split}
					& \iint_{B_r\times E\cap B_{2^{\overline{k}}}^c} \frac{dxdy}{|x-y|^{n+s}} \leq \int_{B_r \times E\cap B_{2^{\overline{k}}}^c} \frac{dxdy}{(|x|-r)^{n+s}} \\
					&\qquad\leq C  \int_{B_{2^{\overline{k}}}^c} \frac{dx}{|x|^{n+s}} \leq C2^{-{\overline{k}}s}. 
				\end{split}
			\end{equation}
			up to relabeling $C$.
			
			Now, plugging \eqref{eq::ude4_Q_min},~\eqref{eq::ude5_1_Q_min},~\eqref{eq::ude5_2_Q_min}, and~\eqref{eq::ude5_3_Q_min} into \eqref{eq::ude3_Q_min}, we find
			\begin{equation*}  
				\mu(r)^{1-\frac{s}{n}} \leq \int_0^r\frac{\mu'(\rho)}{(r-\rho)^{s}}\, d\rho + \int_r^1\frac{\mu'(\rho)}{(\rho-r)^{s}}\, d\rho + \sum_{k=1}^{\overline{k}}2^{-ks}a(2^k) + CQ2^{-\overline{k}s}.
			\end{equation*}
			Now, we integrate la last inequality over $r\in\left(\frac{1}{2},1\right)$,to obtain
			\begin{equation} \label{eq::ude6_Q_min}
				\begin{split}
					&a(1/2)^{1-\frac{s}{n}} \leq C\mu(1/2)^{1-\frac{s}{n}} \leq C\int_{1/2}^{1} \mu(r)^{1-\frac{s}{n}}\, dr\\
					\leq &C\int_{1/2}^{1} \left( \int_0^r\frac{\mu'(\rho)}{(r-\rho)^{s}}\, d\rho + \int_r^1\frac{\mu'(\rho)}{(\rho-r)^{s}}\, d\rho + \sum_{k=1}^{\overline{k}}2^{-ks}a(2^k) + Q2^{-\overline{k}s} \right) dr
				\end{split} 
			\end{equation}
			
			The first term on the right-hand side gives
			\begin{equation} \label{eq::ude6_1_Q_min}
				\begin{split}
					&\int_{1/2}^{1} \int_0^r\frac{\mu'(\rho)}{(r-\rho)^{s}}\, d\rho\, dr = \int_{0}^{1/2}\int_{1/2}^1\frac{\mu'(\rho)}{(r-\rho)^{s}}\, dr\,d\rho + \int_{1/2}^{1}\int_\rho^1\frac{\mu'(\rho)}{(r-\rho)^{s}}\, dr\,d\rho \\ 
					&\qquad \leq C\int_{0}^{1/2}\mu'(\rho)(1-\rho)^{1-s}\, d\rho + C\int_{1/2}^{1}\mu'(\rho)(1-\rho)^{1-s}\, d\rho\\
					&\qquad \leq C\mu(1) \leq Ca(1)
				\end{split}
			\end{equation}
			up to renaming $C>0$. 
			
			Integrating the second term on the right-hand side of~\eqref{eq::ude6_Q_min}, we obtain
			\begin{equation} \label{eq::ude6_2_Q_min}
				\begin{split}
					&\int_{1/2}^{1}\int_r^1\frac{\mu'(\rho)}{(\rho-r)^{s}}\, d\rho \,dr = \int_{1/2}^{1}\int_{1/2}^{\rho} \frac{\mu'(\rho)}{(\rho-r)^{s}}\, dr\,d\rho \\
					&\qquad\leq C\int_{1/2}^{1} \mu'(\rho)\left(\rho-\frac{1}{2}\right)^{1-s}\, d\rho \leq C\mu(1) \leq Ca(1).
				\end{split}
			\end{equation}
			for some $C>0$ depending only on  $n$ and $s$.
			
			Moreover, integrating the last two term on the right-hand side of~\eqref{eq::ude6_Q_min}, we infer that
			\begin{equation} \label{eq::ude6_3_Q_min}
				\int_{1/2}^{1} \left( \sum_{k=1}^{\overline{k}}2^{-ks}a(2^k) + Q2^{-\overline{k}s} \right) dr \leq \frac{1}{2}\sum_{k=1}^{\overline{k}}2^{-ks}a(2^k) + \frac{1}{2}Q2^{-\overline{k}s}.
			\end{equation}

			From~\eqref{eq::ude6_1_Q_min},~\eqref{eq::ude6_2_Q_min}, and~\eqref{eq::ude6_3_Q_min} together with~\eqref{eq::ude6_Q_min}, we deduce that
			\begin{equation} \label{eq::ude7_Q_min}
				\begin{split}
					&a(1/2)^{1-\frac{s}{n}} \leq Ca(1) + C \sum_{k=1}^{\overline{k}}2^{-ks}a(2^k) + CQ2^{-\overline{k}s} \\
					&\qquad \leq C \sum_{k=0}^{\overline{k}}2^{-ks}a(2^k) + CQ2^{-\overline{k}s},
				\end{split} 
			\end{equation}
			up to renaming $C$.
			
			Using~\eqref{eq::small_density_dyadic} and that $2^{-\overline{k}s} \leq \delta_0$ in~\eqref{eq::ude7_Q_min}, we conclude that
			\begin{equation*} 
				\begin{split}
					&a(1/2) \leq \left(C\sum_{k=0}^{\overline{k}}2^{-ks}a(2^{k}) + CQ2^{-\overline{k}s}\right)^{\frac{n}{n-s}} \leq \left(C\overline{k}\delta_0+CQ\delta_0\right)^{\frac{n}{n-s}}\\
					&\qquad\leq \left[C(\overline{k}+Q)\delta_0^{s/n}\right]^{\frac{n}{n-s}}\delta_0 \leq \delta_0,
				\end{split}
			\end{equation*}
			whenever $\delta_0$ is small enough, showing the result.
		\end{proof}
		
		As an immediate consequence of Lemma~\ref{lemma::keep_small_density_Q_min}, we obtain the following
		\begin{cor} \label{cor::keep_small_density_Q_min}
			Let $\delta_0>0$.
			Let~$E$ be a~$Q$-quasi minimal set in $B_{2^{\overline{k}}}$, with $\overline{k}\in\N$ such that $2^{-\overline{k}s}\leq \delta_0$.
			
			If $a(2^{-\overline{k}})\leq 2^{-\overline{k}n}\delta_0$, then $a(1/2)\leq \delta_0$. 
		\end{cor}
		
		The other partial result needed for the proof of Proposition~\ref{prop::UDE_Q_min} is a clean-out lemma.
		
		\begin{lemma} \label{lemma::empty_ball_Q_min}
			Let $\delta_0$ and $\overline{k}$ be as in Corollary~\ref{cor::keep_small_density_Q_min}. Let~$E$ be a~$Q$-quasi minimal set in $B_{2^{\overline{k}}}$ such that
			\begin{equation} \label{eq::small_density_Q_min}
				a(2^{\overline{k}})<\frac{\delta_0}{\omega_n2^{n(\overline{k}+1)}}.
			\end{equation}
			Then, $|E\cap B_{1/2}|=0$.
		\end{lemma}
		
		\begin{proof}
			Arguing by contradiction, suppose that $|E\cap B_{1/2}|>0$. Therefore, there exists a Lebesgue point $x_0\in E\cap B_1$ such that 
			\begin{equation*}
				\lim_{\rho\to0} a(\rho) = 1.
			\end{equation*}
			
			On the other hand, by~\eqref{eq::small_density_Q_min}, we have that also
			\begin{equation} \label{eq::recentered_ball_density_Q_min}
				\frac{|E\cap B_{2^{\overline{k}-1}}(x_0)|}{|B_{2^{\overline{k}-1}}|} \leq \omega_n2^n|E\cap B_{2^{\overline{k}}}| \leq 2^{-n\overline{k}}\delta_0.
			\end{equation}
			
			Without loss of generality, we assume that $x_0=0$, so that~\eqref{eq::recentered_ball_density_Q_min} reads $a(2^{\overline{k}-1})\leq 2^{-n\overline{k}}\delta_0$. Now, thanks to Corollary~\ref{cor::keep_small_density_Q_min}, we also have $a(1/4)\leq \delta_0$.
			Applying inductively Corollary~\ref{cor::keep_small_density_Q_min}, we infer that, for every $k\in\N$,
			\begin{equation*}
				a(2^{-k-2}) \leq \left(C\sum_{j=0}^{\overline{k}}a(2^{-k+j})+CQ2^{-\overline{k}s}\right)^{\frac{n}{n-s}} \leq \left[C(1+Q)\delta_0^{s/n}\right]^{\frac{n}{n-s}}\delta_0 \leq \delta_0.
			\end{equation*}
			
			Therefore,
			\begin{equation*}
				1 = \lim_{k\to+\infty} a(2^{-k-2}) \leq \delta_0,
			\end{equation*}
			which is a contradiction and concludes the proof.
		\end{proof}
		
		\begin{proof}[Proof of Proposition~\ref{prop::UDE_Q_min}]
			Suppose that
			\begin{equation} \label{eq::some_mass_Q_min}
				|E \cap B_\rho|>0,
			\end{equation}
			and let~ $x_0\in( \partial E)\cap B_\rho$, for some small $\rho>0$. Up to a dilations and translation, we suppose that~$x_0$ coincides with the origin. 
			
			Let us recall that, by definition, if~$E$ is a~$Q$-quasi minimal set in $B_1$, for some~$Q\geq1$, then
			\begin{equation*} \label{eq::almost_min2_Q_min}
				\Per_s(E, B_\rho) \leq Q\Per_s(F,B_\rho),
			\end{equation*}
			for every~$F$ such that~$ F\Delta E \subseteq B_\rho \Subset B_1$.
			Observe that quasi minimality is invariant under scalings, and therefore, we have that
			\begin{equation} \label{eq::still_Q_min}
				\Per_s(\rho^{-1}E,B_1) \leq Q\Per_s(\rho^{-1}F,B_1),       
			\end{equation}
			Thus, without loss of generality, we assume $\rho=1$ and that $E$ is $Q$-quasi minimal in $B_{2^{\overline{k}}}$, where $\overline{k}$ is as in Corollary~\ref{cor::keep_small_density_Q_min}. Moreover,~\eqref{eq::some_mass_Q_min} reads
			\begin{equation} \label{eq::some_mass_Q_min2}
				|E \cap B_1|>0.
			\end{equation}
			\smallskip
			We claim that 
			\begin{equation} \label{eq::ude_claim_Q_min}
				|E\cap B_1|\geq \delta_0,
			\end{equation}
			where $\delta_0$ is as in Corollary~\ref{cor::keep_small_density_Q_min}.
			
			Arguing by contradiction, suppose that
			\begin{equation*} \label{eq::ude_small_Q_min}
				|E\cap B_1|\leq \delta_0.
			\end{equation*}
			
			Thus, by mean of Lemma~\ref{lemma::empty_ball_Q_min}, we infer $|E\cap B_1|=0$. However, this contradicts~\eqref{eq::some_mass_Q_min2}, proving~\eqref{eq::ude_claim_Q_min}. Rescaling back by $\rho$, we obtain the first part of~\eqref{eq::unif_dens_estimates_Q_min}.
			\smallskip
			
			In addition, if~$E$ is a quasi minimal set, then also~$E^c$ is quasi minimal . 
			Thus, the argument above with $E^c$ in place of $E$ leads to the second part of~\eqref{eq::unif_dens_estimates_Q_min}, concluding the proof.
		\end{proof}
		
		\subsection*{Optimality of Lipschitz-regularity of $Q$-minimizers} \label{sec::min_graph}
		
		Now, we exhibit a counterexample to H\"older regularity of $Q$-quasi minimal sets. In particular, given $\epsilon>0$, we construct a Lipschitz set $E_\epsilon$ that is $E_\epsilon$ is $Q_\epsilon$-minimal in $B_1$, with $Q_\epsilon\to1$ as $\epsilon\to0$. For this purpose, our key tool will be the following:
		\begin{lemma} \label{lemma::per_preserving}
			Let $\Omega\subseteq\R^n$ an open bounded domain. Let $\epsilon\in(0,1)$ and let $\Phi_\epsilon:\R^n\to\R^n$ be a bi-Lipschitz transformation such that
			\begin{equation} \label{eq::det_control}
				\frac{1}{1+\epsilon} \leq |\det J\Phi_\epsilon| \leq 1+\epsilon. 
			\end{equation}
			
			Then, for any set $F\subseteq \R^n$, we have that
			\begin{equation*}
				\Per_s(F,\Omega) \leq (1+K\epsilon)\Per_s(\Phi_\epsilon(F),\Phi_\epsilon(\Omega)),
			\end{equation*}
			for some $K>0$ depending only on $n$ and $s$.
		\end{lemma}
		\begin{proof}
			Recall that 
			\begin{equation*}
				\Per_s(F,\Omega) = \mathcal{L}_s(F\cap\Omega,F^c\cap\Omega)+\mathcal{L}_s(F\cap\Omega^c,F^c\cap\Omega)+\mathcal{L}_s(F\cap\Omega,F^c\cap\Omega^c).
			\end{equation*}
			Observe that the result will follow if we show that, for any sets $A$ and $B$, we have
			\begin{equation} \label{eq::transformed_interaction}
				\mathcal{L}_s(F\cap A,F^c\cap B) \leq (1+K\epsilon) \mathcal{L}_s(\Phi_\epsilon(F)\cap \Phi_\epsilon(A),\Phi_\epsilon(F^c)\cap \Phi_\epsilon(B)),
			\end{equation}
			for some $K>0$ depending only on $n$ and $s$.
			
			To prove this, To prove this, we use~\eqref{eq::det_control} and the change of variables $(w,z):=(\Phi_\epsilon(x),\Phi_\epsilon(y))$, obtaining
			\begin{equation*}
				\begin{split}
					&\mathcal{L}_s(F\cap A,F^c\cap B) = \int_{F\cap A} \int_{F^c\cap B} \frac{dx\, dy}{|x-y|^{n+s}} \leq (1+\epsilon)^{n+s} \int_{F\cap A} \int_{F^c\cap B} \frac{dx\, dy}{|\Phi_\epsilon(x)-\Phi_\epsilon(y)|^{n+s}}\\
					&\qquad\leq (1+\epsilon)^{n+s+2} \int_{\Phi_\epsilon(F\cap A)} \int_{\Phi_\epsilon(F^c\cap B)} \frac{dw\, dz}{|w-z|^{n+s}} \leq (1+K\epsilon) \int_{\Phi_\epsilon(F\cap A)} \int_{\Phi_\epsilon(F^c\cap B)} \frac{dw\, dz}{|w-z|^{n+s}}.
				\end{split}
			\end{equation*}
			for some $K>0$ depending only on $n$ and $s$.
			
			Moreover, notice that, for any sets $C$ and $D$, it holds
			\begin{equation*}
				\Phi_\epsilon(C\cap D) = \Phi_\epsilon(C) \cap \Phi_\epsilon(D).
			\end{equation*}
			
			Therefore, we infer that
			\begin{equation*}
				\begin{split}
					&\mathcal{L}_s(F\cap A,F^c\cap B)  \leq (1+K\epsilon) \int_{\Phi_\epsilon(F)\cap \Phi_\epsilon(A)} \int_{\Phi_\epsilon(F^c)\cap \Phi_\epsilon(B)} \frac{dw\, dz}{|w-z|^{n+s}} \\
					&\qquad \leq (1+K\epsilon) \mathcal{L}_s(\Phi_\epsilon(F)\cap \Phi_\epsilon(A),\Phi_\epsilon(F^c)\cap \Phi_\epsilon(B)),
				\end{split}
			\end{equation*}		
			establishing~\eqref{eq::transformed_interaction} and concluding the proof.
		\end{proof}
		
		As a corollary, we infer that $Q$-quasi minimizers are generally not better than Lipschitz.
		
		\begin{example} \label{eq::lip_quasi_min}
			Let $\epsilon>0$ chosen conveniently small, and let $E_\epsilon\subseteq\R^2$ be the subgraph of the function $\phi_\epsilon:\R\to\R$ defined as
			\begin{equation*}
				\phi_\epsilon(x) := |x| \tan(\epsilon).
			\end{equation*}
			
			Moreover, we define the bi-Lipschitz transformation $\Phi_\epsilon:\R^n\to\R^n$ as 
			\begin{equation*}
				\Phi_\epsilon(x_1,x_2) := 
				\begin{pmatrix}
					x_1 \\ x_2-\phi_\epsilon(x_1)
				\end{pmatrix}
			\end{equation*}
			Observe that $\Phi_\epsilon(E_\epsilon) = H :=\{x_2<0\}$, $|\det J\Phi_\epsilon| = 1$, and
			\begin{equation*} 
				|\Phi_\epsilon(x_1,x_2)+\Phi_\epsilon(y_1,y_2)| \leq (1+\epsilon)|(x_1,x_2)-(y_1,y_2)|,
			\end{equation*}
			provided that $\epsilon$ is small enough.
			
			Now, let $F_\epsilon\subseteq\R^2$ be any set such that $E_\epsilon\Delta F_\epsilon\subseteq \widetilde{B}_r$, for some ball $\widetilde{B}_r \Subset B_1$. 
			
			Then, employing Lemma~\ref{lemma::per_preserving} and using the minimality of the half-space $H$, we infer that
			\begin{equation*}
				\begin{split}
					&\Per_s(E_\epsilon,B_r) \leq (1+K\epsilon)\Per_s(H,\Phi_\epsilon(B_r)) \\
					&\qquad\leq (1+K\epsilon)\Per_s(\Phi_\epsilon(F_\epsilon),\Phi_\epsilon(B_r)) \leq (1+K\epsilon)\Per_s(F_\epsilon,B_r),
				\end{split}
			\end{equation*}
			for some constant $K>0$ depending only on $s$, and possibly changing from line to line.
			Namely, we have that $E_\epsilon$ is $(1+K\epsilon)$-quasi minimal in $B_1$.
			
			\begin{figure} 
				\centering
				\includegraphics[width=0.8\linewidth]{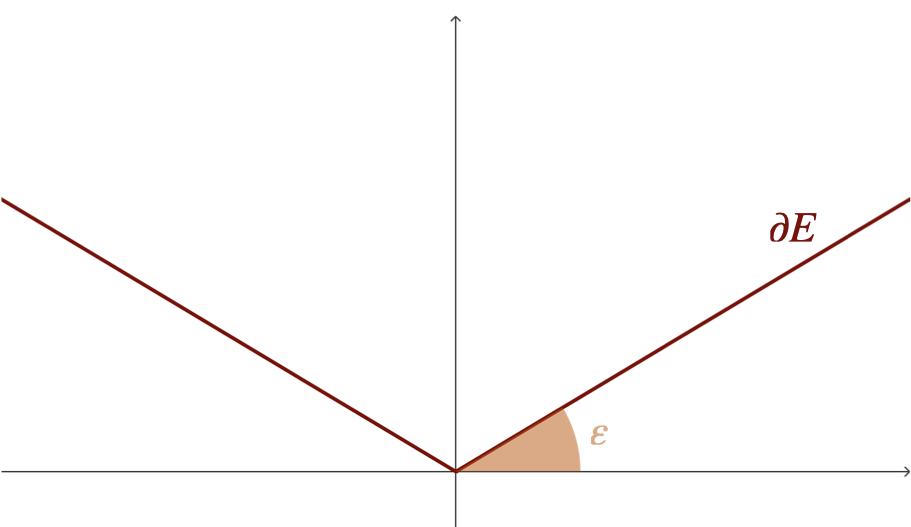}
				\caption{The set $E_\epsilon$.}
			\end{figure}
		\end{example}

		
		\section*{Acknowledgments} \label{sec::acknowledgments}	
		Part of this project was carried out while SD, EV, and RV was visiting Columbia University. For this reason, SD, EV, and RV would like to thank Columbia University for the hospitality.
		
		SD and EV are members of the Australian Mathematical Society (AustMS).
		SD and RV are supported by the Australian Research Council Future Fellowship FT230100333 ``New perspectives on nonlocal equations''.
		OS is supported by NSF Grant DMS-2349794.
		EV is supported by the Australian Laureate Fellowship FL190100081 ``Minimal surfaces, free boundaries and partial differential equations''.
		RV is supported by a Scholarship for International Research Fees at the University of Western Australia.
		
	\bibliographystyle{alpha}

\begin{thebibliography}{CDSS16}
		
		\bibitem[Alm76]{MR420406}
		F.~J. Almgren, Jr.
		\newblock Existence and regularity almost everywhere of solutions to elliptic
		variational problems with constraints.
		\newblock {\em Mem. Amer. Math. Soc.}, 4(165):viii+199, 1976.
		
		\bibitem[BFV14]{MR3331523}
		Bego\~na Barrios, Alessio Figalli, and Enrico Valdinoci.
		\newblock Bootstrap regularity for integro-differential operators and its
		application to nonlocal minimal surfaces.
		\newblock {\em Ann. Sc. Norm. Super. Pisa Cl. Sci. (5)}, 13(3):609--639, 2014.
		
		\bibitem[BLV19]{MR3926519}
		Claudia Bucur, Luca Lombardini, and Enrico Valdinoci.
		\newblock Complete stickiness of nonlocal minimal surfaces for small values of
		the fractional parameter.
		\newblock {\em Ann. Inst. H. Poincar\'e{} C Anal. Non Lin\'eaire},
		36(3):655--703, 2019.
		
		\bibitem[CDSS16]{MR3532394}
		L.~Caffarelli, D.~De~Silva, and O.~Savin.
		\newblock Obstacle-type problems for minimal surfaces.
		\newblock {\em Comm. Partial Differential Equations}, 41(8):1303--1323, 2016.
		
		\bibitem[CL21]{MR4279395}
		Matteo Cozzi and Luca Lombardini.
		\newblock On nonlocal minimal graphs.
		\newblock {\em Calc. Var. Partial Differential Equations}, 60(4):Paper No. 136,
		72, 2021.
		
		\bibitem[CMP15]{MR3401008}
		Antonin Chambolle, Massimiliano Morini, and Marcello Ponsiglione.
		\newblock Nonlocal curvature flows.
		\newblock {\em Arch. Ration. Mech. Anal.}, 218(3):1263--1329, 2015.
		
		\bibitem[CRS10]{MR2675483}
		L.~Caffarelli, J.-M. Roquejoffre, and O.~Savin.
		\newblock Nonlocal minimal surfaces.
		\newblock {\em Comm. Pure Appl. Math.}, 63(9):1111--1144, 2010.
		
		\bibitem[CS07]{MR2354493}
		Luis Caffarelli and Luis Silvestre.
		\newblock An extension problem related to the fractional {L}aplacian.
		\newblock {\em Comm. Partial Differential Equations}, 32(7-9):1245--1260, 2007.
		
		\bibitem[CS09]{MR2494809}
		Luis Caffarelli and Luis Silvestre.
		\newblock Regularity theory for fully nonlinear integro-differential equations.
		\newblock {\em Comm. Pure Appl. Math.}, 62(5):597--638, 2009.
		
		\bibitem[CV13]{MR3107529}
		Luis Caffarelli and Enrico Valdinoci.
		\newblock Regularity properties of nonlocal minimal surfaces via limiting
		arguments.
		\newblock {\em Adv. Math.}, 248:843--871, 2013.
		
		\bibitem[Dav09]{MR2518104}
		Guy David.
		\newblock H\"older regularity of two-dimensional almost-minimal sets in {$\Bbb
			R^n$}.
		\newblock {\em Ann. Fac. Sci. Toulouse Math. (6)}, 18(1):65--246, 2009.
		
		\bibitem[DMV17]{MR3707346}
		Serena Dipierro, Francesco Maggi, and Enrico Valdinoci.
		\newblock Asymptotic expansions of the contact angle in nonlocal capillarity
		problems.
		\newblock {\em J. Nonlinear Sci.}, 27(5):1531--1550, 2017.
		
		\bibitem[DVV26]{MR5015001}
		Serena Dipierro, Enrico Valdinoci, and Riccardo Villa.
		\newblock On non-local almost minimal sets and an application to the non-local
		{M}assari's {P}roblem.
		\newblock {\em J. Math. Pures Appl. (9)}, 208:Paper No. 103840, 78, 2026.
		
		\bibitem[Mir64]{MR0165073}
		Mario Miranda.
		\newblock Distribuzioni aventi derivate misure insiemi di perimetro localmente
		finito.
		\newblock {\em Ann. Scuola Norm. Sup. Pisa Cl. Sci. (3)}, 18:27--56, 1964.
		
		\bibitem[MV17]{MR3717439}
		Francesco Maggi and Enrico Valdinoci.
		\newblock Capillarity problems with nonlocal surface tension energies.
		\newblock {\em Comm. Partial Differential Equations}, 42(9):1403--1446, 2017.
		
		\bibitem[Rud87]{rudin}
		Walter Rudin.
		\newblock {\em Real and complex analysis}.
		\newblock McGraw-Hill Book Co., New York, third edition, 1987.
		
		\bibitem[SS23]{MR4594364}
		Thomas Schmidt and Jule~Helena Sch\"utt.
		\newblock The optimal {H}\"older exponent in {M}assari's regularity theorem.
		\newblock {\em Calc. Var. Partial Differential Equations}, 62(5):Paper No. 161,
		23, 2023.
		
	\end{thebibliography}

\end{document}